\documentclass[12pt, reqno]{amsart}
\usepackage{amssymb, mathrsfs}
\usepackage{latexsym}
\usepackage{amsmath}
\usepackage{amsthm}
\usepackage{amsfonts}
\usepackage{dsfont, upgreek}
\usepackage{mathtools}
\usepackage{epsfig}
\usepackage[nocompress]{cite}
\usepackage{amscd}
\usepackage{graphicx}
\usepackage{tikz-cd}

\usepackage{enumitem}
\setlist[enumerate]{itemsep=0.5ex}

\usepackage{todonotes}

\usepackage{color}
\usepackage{tikz}
\usetikzlibrary{patterns}

\usepackage[left=1.4in,right=1.4in,top=1.2in,bottom=1.2in]{geometry}

\usetikzlibrary{decorations.markings}
\usetikzlibrary{arrows.meta}

\usepackage{extarrows}

\usepackage{adjustbox}
\usepackage{pgfplots}
\usepgfplotslibrary{colormaps}
\usepackage{slashed}

\usepackage[colorlinks=true,linkcolor=blue,citecolor=blue]{hyperref}

\usepackage[all,cmtip]{xy}

\theoremstyle{plain}
\newtheorem{theorem}{Theorem}[section]
\newtheorem{proposition}[theorem]{Proposition}
\newtheorem{lemma}[theorem]{Lemma}
\newtheorem{corollary}[theorem]{Corollary}
\newtheorem{question}[theorem]{Question}
\newtheorem{conjecture}[theorem]{Conjecture}

\theoremstyle{definition} 
\newtheorem{definition}[theorem]{Definition}

\newtheorem*{claim*}{Claim}

\theoremstyle{remark}

\numberwithin{equation}{section}

\newcommand{\N}{\mathbb{N}}

\newcommand{\R}{\mathbb{R}}
\newcommand{\cP}{\mathcal{P}}
\newcommand{\cQ}{\mathcal{Q}}

\newcommand{\Bigwedge}{\mathord{\adjustbox{raise=.4ex, totalheight=.7\baselineskip}{$\bigwedge$}}}

\newcommand{\ind}{\textup{Ind}}
\newcommand{\id}{\mathrm{id}}

\newcommand{\supp}{\mathrm{supp}}

\newcommand{\prop}{\mathrm{prop}}

\newcommand{\sph}{\mathbb{S}}

\newcommand{\ev}{ev}

\newcommand{\interior}[1]{%
	{\kern0pt#1}^{\mathrm{\,o}}%
}

\makeatletter
\let\save@mathaccent\mathaccent
\newcommand*\if@single[3]{%
	\setbox0\hbox{${\mathaccent"0362{#1}}^H$}%
	\setbox2\hbox{${\mathaccent"0362{\kern0pt#1}}^H$}%
	\ifdim\ht0=\ht2 #3\else #2\fi
}
\newcommand*\rel@kern[1]{\kern#1\dimexpr\macc@kerna}
\newcommand*\wideaccent[2]{\@ifnextchar^{{\wide@accent{#1}{#2}{0}}}{\wide@accent{#1}{#2}{1}}}
\newcommand*\wide@accent[3]{\if@single{#2}{\wide@accent@{#1}{#2}{#3}{1}}{\wide@accent@{#1}{#2}{#3}{2}}}
\newcommand*\wide@accent@[4]{%
	\begingroup
	\def\mathaccent##1##2{%
		\let\mathaccent\save@mathaccent
		\if#42 \let\macc@nucleus\first@char \fi
		\setbox\z@\hbox{$\macc@style{\macc@nucleus}_{}$}%
		\setbox\tw@\hbox{$\macc@style{\macc@nucleus}{}_{}$}%
		\dimen@\wd\tw@
		\advance\dimen@-\wd\z@
		\divide\dimen@ 3
		\@tempdima\wd\tw@
		\advance\@tempdima-\scriptspace
		\divide\@tempdima 10
		\advance\dimen@-\@tempdima
		\ifdim\dimen@>\z@ \dimen@0pt\fi
		\rel@kern{0.6}\kern-\dimen@
		\if#41
		#1{\rel@kern{-0.6}\kern\dimen@\macc@nucleus\rel@kern{0.4}\kern\dimen@}%
		\advance\dimen@0.4\dimexpr\macc@kerna
		\let\final@kern#3%
		\ifdim\dimen@<\z@ \let\final@kern1\fi
		\if\final@kern1 \kern-\dimen@\fi
		\else
		#1{\rel@kern{-0.6}\kern\dimen@#2}%
		\fi
	}%
	\macc@depth\@ne
	\let\math@bgroup\@empty \let\math@egroup\macc@set@skewchar
	\mathsurround\z@ \frozen@everymath{\mathgroup\macc@group\relax}%
	\macc@set@skewchar\relax
	\let\mathaccentV\macc@nested@a
	\if#41
	\macc@nested@a\relax111{#2}%
	\else
	\def\gobble@till@marker##1\endmarker{}%
	\futurelet\first@char\gobble@till@marker#2\endmarker
	\ifcat\noexpand\first@char A\else
	\def\first@char{}%
	\fi
	\macc@nested@a\relax111{\first@char}%
	\fi
	\endgroup
}
\makeatother

\newcommand\overbar{\wideaccent\overline}

\makeatletter
\newcommand*{\transpose}{%
	{\mathpalette\@transpose{}}%
}
\newcommand*{\@transpose}[2]{%
	\raisebox{\depth}{$\m@th#1\intercal$}%
}
\makeatother

\makeatletter
\providecommand\@dotsep{5}
\renewcommand{\listoftodos}[1][\@todonotes@todolistname]{%
	\@starttoc{tdo}{#1}}
\makeatother

\begin{document}
\title[$\ell^p$-coarse Baum--Connes conjecture]{$\ell^p$-coarse Baum--Connes conjecture for $\ell^{q}$-coarse embeddable spaces}

\author{Jinmin Wang}
\address[Jinmin Wang]{Institute of Mathematics, Chinese Academy of Sciences}
\email{jinmin@amss.ac.cn}

	\author{Zhizhang Xie}
\address[Zhizhang Xie]{Department of Mathematics, Texas A\&M University }
\email{xie@math.tamu.edu}
\thanks{The second author is partially supported by NSF 1952693 and 2247322.}

\author{Guoliang Yu}
\address[Guoliang Yu]{ Department of
	Mathematics, Texas A\&M University}
\email{guoliangyu@math.tamu.edu}
\thanks{The third author is partially supported by NSF 1700021, 2000082, and the Simons Fellows Program.}

\author{Bo Zhu}
\address[Bo Zhu]{ Department of Mathematics, Texas A\&M University }
\email{zhubomath@gmail.com}

\begin{abstract}
	We prove an $\ell^p$-version of the coarse Baum--Connes conjecture for spaces that coarsely embedds into $\ell^q$-spaces for any $p$ and $q$ in $[1,\infty)$.  
\end{abstract}
\maketitle

\section{Introduction}
The higher index theory is a far-reaching generalization of the Atiyah--Singer index theorem to a large class of not necessarily compact spaces  \cite{willett2020higher}. The higher index of an elliptic differential operator on a complete manifold takes value in $K$-theory of the Roe algebra of the underlying space \cite{Roe,Roecoarse}. The nonvanishing of this higher index serves as an obstruction to the invertibility of the elliptic operator. As an application,  if $(M, g)$ is a complete Riemannian spin manifold such that the higher index of its associated Dirac operator is nonzero, then it follows that $M$ does not support any complete Riemannian metric that is quasi-isometric to $g$ and has uniformly positive scalar curvature. For example, if $N$ is a closed aspherical manifold such that its fundamental group $\pi_1(N)$ has finite asymptotic dimension, then the universal cover of $N$ does not admit a complete Riemannian metric that  is quasi-isometric to $\pi_1(N)$ and has uniformly positive scalar curvature \cite{Yu}.  Consequently, such closed aspherical manifolds themselves do not admit Riemannian metrics of positive scalar curvature. 

In general, computing the higher index of an elliptic differential operator is a difficult task. The coarse Baum–Connes conjecture provides an algorithmic framework for this purpose. Specifically,  for a given metric space $X$, one has the coarse Baum--Connes assembly map $\mu\colon K_\ast(EX) \to K_\ast (C^\ast(X))$,   which maps  each $K$-homology class in the coarse classifying space $EX$ of $X$ to its corresponding higher index class in the $K$-theory of the Roe algebra $C^\ast(X)$ of $X$.   The coarse Baum--Connes conjecture \cite{HigsonRoecBc,YucBC}, which is a coarse geometric analogue of the original Baum--Connes conjecture \cite{BaumConnesHigson}, states that the coarse Baum-Connes assembly map is an isomorphism for any discrete metric space $X$ of bounded geometry.  Recall that a discrete metric space $X$ is said to have bounded geometry if for any $r>0$ there is $N>0$, such that the number of elements in any $r$-ball of $X$ is $\leq N$. See Section \ref{sec:preliminary} for the precise definition of this coarse assembly map. Since $K$-homology can, in principle, be computed using standard tools from homology theory, the coarse Baum--Connes conjecture—when true—yields an effective method for computing both the higher index of elliptic operators and the $K$-theory of the Roe algebra.  We should also mention that the conjectural injectivity of the Baum-Connes coarse assembly map is often called the coarse Novikov conjecture.

Gromov suggested that the coarse embeddability of spaces with bounded geometry is relevant to the Baum--Connes conjecture and  Novikov conjecture and their coarse analogues \cite{Gromov1993asymptotic}. We recall that a metric space $X$ coarsely embedds into a metric space $Y$, if there exists a map $f\colon X\to Y$ together with two non-decreasing function $\rho_\pm\colon \R_{\geq 0}\to \R_{\geq 0}$ such that  $\lim_{x\to\infty}\rho_-(x)=\infty$  and 
$$\rho_-(d_X(x_1,x_2))\leq d_Y(f(x_1),f(x_2))\leq \rho_+(d_X(x_1,x_2)),~\forall x_1,x_2\in X.$$
Indeed, the third author proved that the coarse Baum--Connes conjecture holds for bounded geometry spaces that coarsely embedds into Hilbert space \cite{Yucoarseembed}, generalizing his earlier work on the validity of the coarse Baum--Connes conjecture for bounded geometry spaces with finite asymptotic dimension \cite{Yu}. In general, the surjective of the coarse assembly map holds for relative expanders \cite{MR4585998}, but fails for expanders with large girth \cite{WILLETT20121380}, which does note coarsely embed into $\ell^q$ for any $q\in[1,+\infty)$ \cite{Gromov2010}.

In \cite{MR2980001}, Kasparov and the third author proved the Novikov conjecture for groups that coarsely embed into Banach spaces with property (H). Later, a coarse analogue of this result was proved by  Chen, Wang, and the third author \cite{MR3325537}. Namely, they showed that the coarse Novikov conjecture for bounded geometry spaces that coarsely embeds into Banach spaces with property (H). As a corollary, bounded geometry spaces that coarsely embed into $\ell^q$-spaces satisfy the coarse Novikov conjecture.  
In this paper, we shall prove  the coarse Baum--Connes conjecture for bounded geometry spaces that coarsely embed into $\ell^q$-spaces. 
\begin{theorem}\label{thm:lqcBCIntro}
	Let $(X,d)$ be a discrete metric space with bounded geometry. If  $X$ coarsely embeds into $\ell^q$ for some $q\in[1,\infty)$,  then the coarse Baum--Connes Conjecture holds for $X$.
\end{theorem}

By a theorem of Johnson and Randrianarivony \cite{MR2196037}, one knows that  $\ell^q$   with $q>2$ does not coarsely embed into a Hilbert space. So our main theorem (Theorem \ref{thm:lp-lqcBCIntro}) provides  a strengthening of the result in \cite{Yucoarseembed,MR3325537}. In particular, as a result of the descent principle \cite{Roe}, we obtain a new proof of the Novikov conjecture for discrete groups that coarsely embed into $\ell^q$ \cite{MR2980001}, by taking $X$ as its Cayley graph.

$\ell^p$ techniques play an essential role in Lafforgue's work on the Baum-Connes conjecture  \cite{MR1914617}. In our setting, our approach  yields an $\ell^p$-analogue of the coarse Baum-Connes conjecture. More precisely, this $\ell^p$- analogue  of the coarse Baum--Connes conjecture states that the $\ell^p$-Baum-Connes assembly map 
\[ \mu\colon K_\ast(EX) \to K_\ast (C^p(X)) \]
is an isomorphism for any discrete metric space of bounded geometry $X$, where  $K_\ast(EX)$ is   again the $K$-homology of the coarse classifying space $EX$ of $X$ and  $C^p(X)$ is an $\ell^p$-analogue of the Roe algebra of $X$, which is a Banach algebra acting on an $\ell^p$-space. 

In \cite{Zhangzhoulp}, Zhang and Zhou introduced an $\ell^p$-analogue of Roe algebras and the corresponding $\ell^p$-coarse Baum--Connes conjecture. They proved the $\ell^p$-coarse Baum--Connes isomorphism holds when the space $X$ has finite asymptotic dimension (thus coarsely embeddable into Hilbert space). In \cite{MR4357067}, Lin and Wang proved an $\ell^p$-coarse Novikov conjecture for spaces coarsely embeddable  into a Hadamard manifold. In \cite{MR4832450}, Zhang proved an $\ell^p$-coarse Baum--Connes conjecture for certain metric cones. We should remark that the $\ell^p$-coarse Baum--Connes conjecture fails in general. For example, Chung and Nowak showed that certain expanders are counterexamples to the $\ell^p$-coarse Baum--Connes conjecture \cite{MR4565436}.

In this paper, we introduce a slightly different definition of $\ell^p$-Roe algebras (see Definition \ref{def:lpRoe} and the discussion below it), and prove the corresponding $\ell^p$-coarse Baum--Connes conjecture.

\begin{theorem}\label{thm:lp-lqcBCIntro}
	Let $(X,d)$ be a discrete metric space with bounded geometry. If  $X$ coarsely embeds into $\ell^q$ for some $q\in[1,\infty)$,  then for any $p\in[1,\infty)$, the $\ell^p$-coarse Baum--Connes Conjecture (see Conjecture \ref{conj:l^pCBC}) holds for $X$.
\end{theorem}

Our strategy for proving Theorem \ref{thm:lp-lqcBCIntro} is inspired by the work of  \cite{ MR3325537, MR2980001, willett2020higher}. In particular, we use the coarse embedding of $X$ into $\ell^q$ and the Mazur's map (see line \eqref{eq:Mazur} for the precise definition) to construct a Bott--Dirac operator acting on finite dimensional Euclidean spaces. We twist each $K$-theory element of the $\ell^p$-Roe algebra $C^p(X)$ with this Bott--Dirac operator. The resulting element acts on a Banach space equipped with a norm that mixes $\ell^p$ and $\ell^2$ norms. A key ingredient to deal with the $\ell^p$-$\ell^2$-mixed norm is a vector-valued version of the  Marcinkiewicz--Zygmund inequality (Theorem \ref{thm:VectorMarZygmundIneq}), which itself is proved by using the Grothendieck inequality. Ultimately, we prove Theorem \ref{thm:lp-lqcBCIntro} by  a Mayer--Vietoris argument, where the cutting-and-pasting is carried out with respect to the $\ell^q$-metric.

This paper is organized as follows. In Section \ref{sec:lqCBC}, we recall the classic definition of the $C^*$-version of Roe algebras and the coarse Baum--Connes map, and prove Theorem \ref{thm:lp-lqcBCIntro} for $p=2$. We focus on the geometry of the embedding into $\ell^q$-space, deferring the discussion of the $\ell^p$-Banach algebras. We will introduce the $\ell^p$-Roe algebras and the $\ell^p$-coarse Baum--Connes conjecture in Section \ref{sec:lplqCBC} and prove Theorem \ref{thm:lp-lqcBCIntro} in full generality. 

\section{Coarse Baum--Connes conjecture for coarsely embeddable spaces}\label{sec:lqCBC}
In this section, we focus on the $C^*$-algebra version of coarse Baum--Connes conjecture for spaces that coarsely embedds into $\ell^q$-spaces for $q\in[1,\infty)$. Our proof is inspired by the work from  \cite{Yucoarseembed,willett2020higher}.
\begin{definition}\label{def:coarseembed}
	Suppose that $(X,d)$ is a proper metric space and $\ell^q=\ell^q(\mathbb N)$.  $X$ is said to coarsely embed into $\ell^q$ if there exists a map $f\colon X\to \ell^q$  and  $\rho_\pm\colon \R_{\geq 0}\to\R_{\geq 0}$  such that
	\begin{enumerate}
		\item  $\lim_{x\to\infty}\rho_-(x)=+\infty$  and
		
		\item $\rho_-(d(x,y))\leq \|f(x)-f(y)\|_q\leq \rho_+(d(x,y))$.
	\end{enumerate}
\end{definition}

The main theorem of this section is as follows.
\begin{theorem}\label{thm:lqcBC}
	Let $(X,d)$ be a proper metric space that $X$ coarsely embeds into $\ell^q$ for some $q\in[1,\infty)$. Then the coarse Baum--Connes Conjecture \ref{conj:cBC} holds for $X$.
\end{theorem}

By \cite[Theorem 5]{Nowakcoarselp}, we know that the coarse embeddability into $\ell^q$ when $1\leq q \leq 2$ is equivalent to the coarse embeddability into the Hilbert space $\ell^2$. Therefore  we assume that $q\geq 2$ in the following discussion. 

\subsection{Geometric $C^*$-algebras and coarse Baum--Connes conjecture}\label{sec:preliminary}

In this subsection, we recall the definitions of various geometric $C^*$-algebras. We refer the  reader to \cite{willett2020higher} for more details.

\begin{definition}\label{def:Roe} 	Let $X$ be a proper metric space. An infinite dimensional separable Hilbert space $H$ is called a geometric $X$-module if $H$ is equipped with a non-degenerate representation of $C_0(X)$, and no non-zero element of $C_0(X)$ acts as a compact operator. Let $T$ be a bounded operator acting on $H$.
	\begin{enumerate}
		\item $T$ is called \emph{locally compact} if for any compact subset $K$ of $X$, $\chi_K T$ and $T\chi_K$ are compact operators, where $\chi_K$ is the characteristic function of $K$.
		\item The \emph{support} of $T$, denoted by $\supp(T)$, consists of all points $(y,x)\in X\times X$ such that for all open neighborhood $U$ of $x$ and $V$ of $y$, $\chi_VT\chi_U\ne 0$.
		\item The \emph{propagation} of $T$, denoted by $\prop(T)$, is defined by the supremum of $d(x,y)$ for any pair $(x,y)\in \supp(T)$.
		\item The collection of all locally compact operators with finite propagation over $H$ is denoted by $ C^*_{alg}(X)$. The completion of $ C^*_{alg}(X)$ with respect to the operator norm acting on $H$ is called the Roe algebra $ C^*(X)$.
	\end{enumerate}
\end{definition}

Let $X$ be a locally finite metric space. For each $r>0$, we define the Rips complex $P_r(X)$ to be the simplicial complex generated by $X$, such that $x_i$ and $x_j$ are in the same simplex if $d(x_i,x_j)\leq r$. By construction, $P_r(X)$ is finite dimensional. We equip $P_r(X)$ with the spherical metric: for each simplex
$$\left\{\sum_{k=1}^m t_kx_{i_k}:\sum_{k=1}^m t_k=1,~t_k\geq 0\right\},$$
its metric is the one obtained from the sphere $\sph^m$ through the following map:
$$\sum_{k=1}^m t_kx_{i_k}\mapsto\left(\frac{t_0}{\sqrt{\sum_{k=1}^m t_k^2}},\cdots,\frac{t_m}{\sqrt{\sum_{k=1}^m t_k^2}}\right).$$

It follows from \cite[Section 5.1]{willett2020higher} that the $K$-theory of Roe algebras is invariant under coarse equivalence. In particular, the $K$-theory of the Roe algebras of $X$ and its Rips complexes are canonically isomorphic.

Now we recall the definition of localization algebras for simplicial complexes.
\begin{definition}
	Let $P$ be a finite dimensional simplicial complex and $Z$ a countable dense subset of $P$. We take $\ell^2(Z)\otimes \ell^2(\mathbb N)$ as a geometric $P$=module. Let $C^*_L(P)$ be the norm completion of all maps $T\colon [0,\infty)\to  C^*_{alg}(Z)$ such that
	\begin{enumerate}
		\item for each $t\geq 0$, $T(t)$ has finite propagation,
		\item $\text{prop}(T(t))\to 0$ as $t\to\infty$, and
		\item $t\mapsto T(t)$ is uniformly continuous and uniformly bounded with respect to the operator norm,
	\end{enumerate}
	with respect to the sup-norm
	$$\|T\|\coloneqq \sup_{t\geq 0}\|T(t)\|.$$
\end{definition}

In general, the $K$-theory of $ C^*_L(P)$ is naturally isomorphic to the $K$-homology of $P$, which follows from a Mayer--Vietoris argument \cite{Yulocalization}.

Now we formulate the $C^*$-algebra version of coarse Baum--Connes conjecture. 
\begin{conjecture}[Coarse Baum--Connes Conjecture]\label{conj:cBC}
Let $X$ be a locally finite metric space with bounded geometry.  The coarse Baum--Connes conjecture for $X$ states that the evaluation map 
	$$\ev\colon C^*_L(P_r(X))\to  C^*(X)$$
	induces an isomorphism 
	$$\ev_*\colon \lim_{r\to\infty}K_*( C^*_L(P_r(X)))\to K_*( C^*(X)).$$
\end{conjecture}

The injectivity part of the coarse Baum--Connes conjecture is usually called the coarse strong Novikov conjecture.


\subsection{$\ell^q$-Bott--Dirac operator}\label{sec:BottDirac}
In this subsection, we study the Bott--Dirac operator constructed from the $\ell^q$-norm.

Let $E$ be a vector space of dimension $N$ equipped with a Euclidean $\ell^2$-norm.   Let $\Bigwedge^*E$ be the exterior algebra of $E$, equipped with its natural Euclidean metric. Consider the Clifford multiplication operators  on $\Bigwedge^*E$ given by
$$c(v)\omega=v\wedge \omega+v\lrcorner \omega $$
$$\overbar c(v)\omega=v\wedge \omega -v\lrcorner \omega$$
for $v\in E$ and $\omega \in \Bigwedge^\ast E$, where $\lrcorner$ is the standard interior multiplication.  Note that if $v$ is a unit vector, then $c(v)^2 = 1$, $\overbar c(v)^2 = -1$, and $c(v)\overbar c(v) + \overbar c(v) c(v) = 0$.

Suppose that $e_1,e_2,\cdots,e_N$ form an orthonormal basis of $(E, \|\cdot\|_2)$. We consider the $\ell^q$-norm $\|\cdot\|_q$ on $E$ given by
$$\left\|\sum_{i=1}^n x_ie_i\right\|_q=\left(\sum_{i=1}^n|x_i|^q\right)^{\frac 1 q}.$$ We denote by $S_q(v,R)$ the sphere in $(E, \|\cdot \|_q)$ of radius $R$ centered at $v$, and $B_q(v,R)$ the ball in $(E, \|\cdot \|_q)$ of radius $R$ centered at $v$. In the case where $v=0$, we shall omit $0$ and simply write $S_q(R)$ and $B_q(R)$ to denote the  $R$-sphere and $R$-ball in $(E, \|\cdot \|_q)$ centered at the origin. Recall the Mazur map $S_q(1)\to S_2(1)$ given by 
\begin{equation}\label{eq:Mazur}
	\psi\colon v=\sum_{i=1}^N v_ie_i\longmapsto \psi(v)=\sum_{i=1}^N v_i|v_i|^{q/2-1}e_i.
\end{equation}
The Mazur map is Lipschitz if $q\geq 2$, and $q/2$-H\"{o}lder continuous if $q\leq 2$. More precisely, we have
$$\|\psi(u)-\psi(v)\|_2\leq 2\left(\|u-v\|_q\right)^{q/2},~\text{ if }q\leq 2,$$
and
$$\|\psi(u)-\psi(v)\|_2\leq q2^{q-1}\|u-v\|_q,~\text{ if }q\geq 2.$$
See for example \cite[Theorem 12]{MR1188900}.
Now we define $\Psi\colon (E, \|\cdot\|_q) \to (E, \|\cdot\|_2)$ to be the homogeneous extension of $\psi$, namely
$$\Psi(v)=\|v\|_q\cdot \psi\left(v\|v\|_q^{-1}\right)$$
for all $0\neq v\in (E, \|\cdot\|_q)$ and $\Psi(0) =0$. 
\begin{lemma}\label{lemma:lipschitzBoth}
	Assume $q\geq 2$. With the above notation, we have $\|\Psi(v)\|_{2}=\|v\|_{q}$ for all $v\in (E, \|\cdot\|_q)$. Furthermore, $\Psi$ is Lipschitz, that is, 
	$$\|\Psi(v)-\Psi(u)\|_2\leq L\|v-u\|_q,$$
	for any $u, v\in (E, \|\cdot\|_q)$, where $L=1+q\cdot 2^{q}$.
\end{lemma}
\begin{proof}
	Clearly, we have 
	$\|\Psi(v)\|_2=\|v\|_q$. 
	Let us now show that 
	$$\|\Psi(v)-\Psi(u)\|_2\leq L\|v-u\|_q.$$
	Without loss of generality, we assume that $u$ and $v$ are both non-zero.
	For any $r>0$, we denote $v_r=r\|v\|_q^{-1}\cdot v$, namely the vector that has length $r$ and the same direction as $v$. Write $\|u\|_q=s$ and $\|v-u\|_q=\delta$. We note that
	$$\|\Psi(v)-\Psi(u)\|_2\leq \|\Psi(v)-\Psi(v_s)\|_2+\|\Psi(v_s)-\Psi(u)\|_2.$$
We estimate the two terms on the right hand side separately. For the first term, we have 
	\begin{align*}
		\|\Psi(v)-\Psi(v_s)\|_2=&\|(\|v\|_q-s)\psi(v_1)\|_2=|\|v\|_q-s|\\
		=&|\|v\|_q-\|u\|_q|\leq \|v-u\|_q=\delta.
	\end{align*}
For the second term, we first note that $\|v_s\|_q=\|u\|_q=s$ and
$$\|v_s-v\|_q=\|(s-\|v\|_q)\cdot v_1\|_q=|\|u\|_q-\|v\|_q|\leq \delta, $$
which implies that 
$$\|v_s-u\|_q\leq\|v_s-v\|_q+\|v-u\|_q\leq 
2\delta.$$
It follows that 
\begin{align*}
	\|\Psi(v_s)-\Psi(u)\|_2&= s\|\psi(s^{-1} v_s)-\psi(s^{-1}u)\|_2\leq s\cdot q2^{q-1}\|s^{-1}v_s-s^{-1}u \|_q\\
	 &\leq s\cdot q2^{q-1}\cdot \frac{2\delta}{s}=q2^{q}\cdot\delta, 
\end{align*}
since $q\geq 2$. 
 To summarize, we have
$$\|\Psi(v)-\Psi(u)\|_2\leq \delta+q2^{q}\cdot\delta=L\delta=L\|v-u\|_q.$$
This finishes the proof.
\end{proof}

 We emphasize that the constant $L$ in Lemma \ref{lemma:lipschitzBoth} is independent of the dimension of $E$.

 For any $v\in E$, let $C_v$ be  the Clifford multiplication operator given   by
$$C_v(w)= c(\Psi(w-v))$$
for all $w\in E$.
Let 
 \[  D=d+d^* = \sum_{i=1}^{N} \overbar c(e_i) \frac{\partial}{\partial x_i}  \] 
 be the de Rham operator, 
 where $e_1, \dots, e_N$ are the basis vectors of $(E, \|\cdot\|_2)$ fixed as before and $x_1, \dots, x_N$ are the corresponding coordinates of $E$.

For any $s>0$, we define the following analogue of the  Bott--Dirac operator: 
\begin{equation}
	B_{s,v}=s^{-1}D+C_v
\end{equation}
acting as an unbounded operator on the Hilbert space $L^2(E,\Bigwedge^*E)$. Here $L^2(E,\Bigwedge^*E)$ the Hilbert space of $L^2$-integrable functions from  $(E, \|\cdot\|_2)$ to $(\Bigwedge^*E, \|\cdot\|_2)$.   
\begin{lemma}\label{lemma:Fredholm}
	If $q\geq 2$, then for any $s\geq 1$, $B_{s,v}$ is a self-adjoint Fredholm operator and has odd grading with respect to the even-odd grading on $\Bigwedge^*E$. Moreover, its Fredholm index is $1$. 
\end{lemma}
\begin{proof}
	We may assume  without loss of generality that $v=0$ and omit the subscript $v$ from now on. Note that
	\begin{align*}
		B_{s}^2=s^{-2}D^2+s^{-1}[D,C]+C^2,
	\end{align*}
where 
$$C(x)^2=\|\Psi(x-0)\|_2^2=\|x\|_q^2.$$
Let us write $\Psi(x)=(\Psi_1(x),\ldots,\Psi_N(x))$, and
$C(x)=\sum_{i=1}^N \Psi_i(x) c(e_i)$. We have
$$[D,C]=\sum_{i,j=1}^N\frac{\partial \Psi_i(x)}{\partial x_j}\overbar c(e_j)c(e_i) .$$
Since $\Psi$ is $L$-Lipschitz, we have $\Psi_i=\langle\Psi,e_i\rangle$ is $L$-Lipschitz, hence
$$\|[D,C]\|\leq N^2L.$$
Let $\rho_R$ be any $1$-Lipschitz function on $E$ that is supported outside $B_q(R)$. For any $\varphi\in C_c^\infty(E,\Bigwedge^*E)$, we have
$$\|B_s(\rho_R\varphi)\|^2=\langle B_s^2(\rho_R\varphi),\rho_R\varphi\rangle\geq\frac 1 4\|R^2\rho_R\varphi\|^2$$
if $R$ is large enough. Set
$$\sigma=\left(1+B_s^2\right)^{1/2}\varphi.$$	
It follows that
\begin{align*}
	&\|\rho_R (1+B_s^2)^{-1/2}\sigma\|\leq \frac{2}{R}\|B_s\rho_R(1+B_s^2)^{-1/2}\sigma\|\\
	\leq&\frac 2 R\left(\|[B_s,\rho_R](1+B_s^2)^{-1/2}\|+\|\rho_RB_s(1+B_s^2)^{-1/2}\|\right)\|\sigma\|\leq \frac{2}{R}(1+s^{-1})\|\sigma\|.
\end{align*}
Here we have used the fact that
$$\|[B_s,\rho_R]\|=\|s^{-1}[D,\rho_R]\|=s^{-1}\|c(d\rho_R)\|\leq s^{-1}.$$
As a result, for any fixed $s>0$, we have
$$\left\|(1+B_s^2)^{-1/2}-(1-\rho_R)(1+B_s^2)^{-1/2}\right\|\to 0$$
as $R\to\infty$.  Note that by the Rellich lemma, $(1-\rho_R)(1+B_s^2)^{-1/2}$ is a compact operator. Therefore 
$\left(1+B_{s}^2\right)^{-1/2}$ is compact. It follows that $B_s$ is Fredholm

 To compute the index of $B_s$, we consider the standard Bott potential given by 
\[   C^{\beta}(x) = \sum_{i=1}^N x_i c(e_i).\] 
We connect the potential $C$ to $C^\beta$ via a linear path which we denote by $C^t, t\in [0, 1]$. The same argument above shows that the operators $ s^{-1}D + C^t$ are Fredholm. In particular, it follows that the Fredholm index of $B_s$ coincides with the Fredholm index of $s^{-1}D+ C^\beta$, the latter of which  is $1$ by a standard computation.
\end{proof}

\begin{definition}
	For any $s\geq 1$ and $v\in E$, given $B_{s,v}=s^{-1}D+C_v$, we define
	$$F_{s,v}=B_{s,v}\left(1+B_{s,v}^2\right)^{-1/2}.$$
\end{definition}


Let us first fix some notations.
\begin{definition}
	Let $F$ be a bounded operator on $L^2(E,\Bigwedge^*E)$. 
	\begin{enumerate}
		\item The \emph{$E$-support} of $F$ is defined to be the smallest closed set in $E$, denoted by $\supp_E(F)$, such that if $A$ is an open set away from $\supp_E(F)$, then $\chi_AF=F\chi_A=0$, where $\chi_A$ is the characteristic function of $A$.
		\item The $\ell^q$-propagation of $F$ is defined to be the smallest number, denoted by $\prop_{q}(F)$, such that if $A,B$ are two open sets in $E$ that are $d$-away in $\ell^q$-distance with $d>\prop_{q}(F)$, then $\chi_AF\chi_B=0$.
	\end{enumerate}
\end{definition}

For example, the $E$-support of a integral kernel operator
$$\varphi\mapsto\int_E K(x,y)\varphi(y)dy$$
is equal to the union of $\pi_1\supp(K)$ and $\pi_2\supp(K)$, where $\pi_i$ is the projection from $E\times E$ to its $i$th component.

 The following lemma summarizes some useful consequences  of the fact that  $\|\cdot\|_q\leq\|\cdot\|_2$ for $q\geq 2$.
\begin{lemma}\label{lemma:q-norm<=2-norm}
Assume that $q\geq 2$. Then the following hold. 
	\begin{enumerate}
		\item Any $\ell^q$-$L$-Lipschitz function on $E$ is also $\ell^2$-$L$-Lipschitz. 
		\item For any bounded operator on $L^2(E,\Bigwedge^*E)$, its $\ell^q$-propagation is no more than its $\ell^2$-propagation.
	\end{enumerate}
\end{lemma} 
\begin{proof}
	For the convenience of the reader, we sketch a proof. Let $\{e_i\}$ be the basis chosen above. Suppose that $v=\sum_{i=1}^N v_ie_i$. We have 
	$$(\|v\|_q)^q=\sum_{i=1}^N |v_i|^q=\sum_{i=1}^N (|v_i|^2)^{q/2}\leq \left(\sum_{i=1}^N |v_i|^2\right)^{q/2}=(\|v\|_2)^q.$$

	To prove (1), assume that $h$ is an $\ell^q$-$L$-Lipschitz function. It follows that
	$$|h(v)-h(u)|\leq L\|v-u\|_q\leq L\|v-u\|_2,$$
	which shows that $h$ is $\ell^2$-$L$-Lipschitz.
	
	To prove (2), let $\prop_2(F)$ be the $\ell^2$-propagation of $F$. Consider open sets $A,B$ that are $d$-away in $\ell^q$-norm with $d>\prop_2(F)$, namely $\|x-y\|_q\geq d$ for any $x\in A$ and $y\in B$. It follows that $\|x-y\|_2\geq \|x-y\|_q\geq d$, namely $A$ and $B$ are also $d$-away in $\ell^2$-norm with $d>\prop_2(F)$. Therefore we have $\chi_AF\chi_B=0$. By definition, we conclude that $\prop_q(F)\leq\prop_2(F)$.
\end{proof}

Now let us prove  some key properties of the operator $F_{s,v}$.
\begin{lemma}\label{lemma:cut}
	Suppose that $q\geq 2$. Let $\chi_{v,R}$ be the characteristic function of $B_q(v,R)$. If $s\geq 2N^2L$, then 
	$$\| (1-\chi_{v,R})\left(1+\lambda^2+B_{s,v}^2\right)^{-1/2}\|\leq\frac{1}{\sqrt{1/2+\lambda^2+R^2}},$$
	for any $\lambda>0$.
\end{lemma}
\begin{proof}
	We may  assume without loss of generality that $v=0$, thus omit the base point $v$ from the notation. 
	By the proof of Lemma \ref{lemma:Fredholm}, we see that if $s\geq 2N^2L$, then
	$$\langle\left(1+\lambda^2+B_s^2\right)\varphi,\varphi\rangle\geq \langle\left(1/2+\lambda^2+\|x\|_q^2\right)\varphi,\varphi\rangle$$
	for any $\varphi\in C_c^\infty(E_q,\Bigwedge^*E)$. Set
	$$\sigma=\left(1+\lambda^2+B_s^2\right)^{1/2}\varphi.$$	
	It follows that
	$$\|\sigma\|^2\geq\left\|\sqrt{1/2+\lambda^2+\|x\|_q^2}\left(1+\lambda^2+B_s^2\right)^{-1/2}\sigma\right\|^2.$$
	Thus we have
	\begin{align*}
		\|\sigma\|^2\geq &\int_E(1/2+\lambda^2+\|x\|_q^2)\cdot\left\|\left(1+\lambda^2+B_s^2\right)^{-1/2}\sigma\right\|^2dx\\
		\geq&\int_E(1/2+\lambda^2+\|x\|_q^2)(1-\chi_{R})\cdot \left\|\left(1+\lambda^2+B_s^2\right)^{-1/2}\sigma\right\|^2dx\\
		\geq&\int_E(1/2+\lambda^2+R^2)\cdot\left\|(1-\chi_{R})\left(1+\lambda^2+B_s^2\right)^{-1/2}\sigma\right\|^2dx\\
		=&(1/2+\lambda^2+R^2)\left\|(1-\chi_{R})\left(1+\lambda^2+B_s^2\right)^{-1/2}\sigma\right\|^2.
	\end{align*}
	It follows that 
$$\| (1-\chi_{v,R})\left(1+\lambda^2+B_{s,v}^2\right)^{-1/2}\|\leq\frac{1}{\sqrt{1/2+\lambda^2+R^2}},$$
as $\left(1+\lambda^2+B_s^2\right)^{1/2}$ is surjective. 
	This finishes the proof.
\end{proof}

As $F_{s,v}^2-1=(1+B_{s,v}^2)^{-1}$, we have the following immediate corollary.
\begin{corollary}\label{coro:cut}
	Suppose that $q\geq 2$. Let $\chi_{v,R}$ be the characteristic function of $B_q(v,R)$. If $s\geq 2N^2L$, then 
	$$\| (1-\chi_{v,R})\left(F_{s,v}^2-1\right)\|\leq\frac{1}{\sqrt{1/2+R^2}}.$$
\end{corollary}

\begin{lemma}\label{lemma:difference}
	Suppose that $q\geq 2$. Then $F_{s,v}-F_{s,u}$ is compact, and
	$$\|F_{s,v}-F_{s,u}\|\leq\frac{5L}{4}\|v-u\|_q.$$
\end{lemma}
\begin{proof}
	 Note that
	$$F_{s,v}=\frac 2\pi\int_0^\infty B_{s,v}\left(1+\lambda^2+B^2_{s,v}\right)^{-1}d\lambda,$$
	where the integral converges in the strong operator topology.
	
	Set
	$$ c_{v,u}(w)=B_{s,v}(w)-B_{s,u}(w)= c(\Psi(w -v)-\Psi(w-u))$$
	for all $w\in E$. Note that
	\begin{align*}
		&B_{s,v}\left(1+\lambda^2+B_{s,v}^2\right)^{-1}-B_{s,u}\left(1+\lambda^2+B_{s,u}^2\right)^{-1}\\
		=&\left(1+\lambda^2+B_{s,v}^2\right)^{-1}\left( (1+\lambda^2) c_{v,u}
		-B_{s,v} c_{v,u}B_{s,u} \right)
		\left(1+\lambda^2+B_{s,u}^2\right)^{-1}.
	\end{align*}
Since $\Psi$ is $L$-Lipschitz, we have
$$\|c_{v,u}(w)\|\leq\sup_{w\in E}\|\Psi(w -v)-\Psi(w-u)\|_2\leq L\|v-u\|_q.$$
Therefore, we obtain that
$$\|\left(1+\lambda^2+B_{s,v}^2\right)^{-1}(1+\lambda^2) c_{v,u}
\left(1+\lambda^2+B_{s,u}^2\right)^{-1}\|\leq \frac{L\|v-u\|_q}{1+\lambda^2},$$
and
$$\|\left(1+\lambda^2+B_{s,v}^2\right)^{-1}B_{s,v} c_{v,u}B_{s,u}
\left(1+\lambda^2+B_{s,u}^2\right)^{-1}\|\leq \frac{L\|v-u\|_q}{4(1+\lambda^2)}$$
Thus 
$$\left\|B_{s,v}\left(1+\lambda^2+B_{s,v}^2\right)^{-1}-B_{s,u}\left(1+\lambda^2+B_{s,u}^2\right)^{-1}\right\|\leq
\frac{5L\|v-u\|_q}{4(1+\lambda^2)},
$$
In particular, $B_{s,v}\left(1+\lambda^2+B_{s,v}^2\right)^{-1}-B_{s,u}\left(1+\lambda^2+B_{s,u}^2\right)^{-1}$ is compact for each $\lambda\in [0, \infty)$, and the integral 
\[\int_{0}^\infty B_{s,v}\left(1+\lambda^2+B_{s,v}^2\right)^{-1}-B_{s,u}\left(1+\lambda^2+B_{s,u}^2\right)^{-1}d\lambda \] converges absolutely. It follows that $F_{s,v}-F_{s,u}$ is compact. Furthermore,
$$\|F_{s,v}-F_{s,u}\|\leq\frac2\pi\int_0^\infty\frac{5L\|v-u\|_q}{4(1+\lambda^2)}d\lambda=\frac{5L}{4}\|v-u\|_q.$$
This finishes the proof.
\end{proof}

By combining Lemma \ref{lemma:cut} and the proof of Lemma \ref{lemma:difference}, we have the following corollary.
\begin{corollary}\label{coro:differenceCut}
	For any $\varepsilon>0$ and $r>0$, there exists $R(r,\varepsilon)>0$ such that if $R>R(r,\varepsilon)$, then $\|(1-\chi_{v,R})(F_{s,v}-F_{s,u})\|\leq \varepsilon$ for all $s\geq 2N^2L$ and  all $v,u\in E$ with $\|v-u\|_q<r$, where $\chi_{v,R}$ is the characteristic function of $B_q(v,R)$.
\end{corollary}
\begin{proof}
By following the proof of Lemma \ref{lemma:cut}, we have 
\begin{align}
	&(1-\chi_{v,R})(F_{s,v}-F_{s,u}) \notag\\
	=&\int_0^\infty (1-\chi_{v,R}) \left(1+\lambda^2+B_{s,v}^2\right)^{-1}(1+\lambda^2) c_{v,u}
	\left(1+\lambda^2+B_{s,u}^2\right)^{-1} d\lambda \notag \\
	& + \int_0^\infty (1-\chi_{v,R}) \left(1+\lambda^2+B_{s,v}^2\right)^{-1}B_{s,v} c_{v,u}B_{s,u}
	\left(1+\lambda^2+B_{s,u}^2\right)^{-1} d\lambda \label{eq:estimate}.
\end{align}
By the same proof of Lemma \ref{lemma:difference}, we have
$$\| (1-\chi_{v,R})\left(1+\lambda^2+B_{s,v}^2\right)^{-1}\|\leq\frac{1}{1/2+\lambda^2+R^2},$$
for all $s\geq 2N^2L$ and $\lambda>0$. This together with the fact 
\[ \|(1+\lambda^2) c_{v,u}
\left(1+\lambda^2+B_{s,u}^2\right)^{-1}\| \leq L\|v-u\|_q  \]
takes care of the first term on the right hand side of \eqref{eq:estimate}. Now for the second term on the right hand side of \eqref{eq:estimate}, we observe that the integrand is the composition of three operators
\[  (1-\chi_{v,R}) \left(1+\lambda^2+B_{s,v}^2\right)^{-1/2}, \left(1+\lambda^2+B_{s,v}^2\right)^{-1/2} B_{s, v} \textup{ and } c_{v,u}B_{s,u}
\left(1+\lambda^2+B_{s,u}^2\right)^{-1},\]
which satisfy
\[ \|(1-\chi_{v,R}) \left(1+\lambda^2+B_{s,v}^2\right)^{-1/2} \| \leq  \frac{1}{\sqrt{1/2+\lambda^2+R^2}},  \]
\[ \|\left(1+\lambda^2+B_{s,v}^2\right)^{-1/2} B_{s, v}\|\leq 1, \]
and 
\[ \|B_{s,u}
\left(1+\lambda^2+B_{s,u}^2\right)^{-1}\| \leq \frac{1}{2\sqrt{1+\lambda^2}}. \]
These estimates take care of the second term on the right hand side of \eqref{eq:estimate}. This finishes the proof. 
\end{proof}

It is easy to see that Lemma \ref{lemma:cut}, Corollary \ref{coro:cut}, and Corollary \ref{coro:differenceCut} also hold if the cut-off functions $\chi_{v,R}$ is  replaced by a  smooth cut-off function. In the following, we require these smooth cut-off functions to  have small Lipschitz constants. In particular,  we have the following lemma. 
\begin{lemma}\label{lemma:cut-off}
	For any $R>0$, there exists a smooth function $\psi_{0,R}$ on $E$ such that $\psi_{0,R}\equiv 1$ in $B_q(R/2)$,  $\psi_{0,R}\equiv 0$ in $E-B_q(R)$, and $\psi_{0,R}$ is at most $3R^{-1}$-Lipschitz with respect to both $\ell^q$-norm and $\ell^2$-norm. Furthermore, if we set
	$$\psi_{v,R}(w)=\psi_{0,R}(w-v),~\forall v\in E,$$
	then $\sup_{w\in E}|\psi_{v,R}(w)-\psi_{u,R}(w)|\leq 3R^{-1}\|v-u\|_q$.
\end{lemma}
\begin{proof}
	Note that the $\ell^q$-distance between $B_q(R/2)$ and $E-B_q(R)$ is $\geq R/2$. Thus we may choose the cut-off function $\psi_{0,R}$ to  have $\ell^q$-Lipschitz constant $\leq 3R^{-1}$. Consequently, its $\ell^2$-Lipschitz constant is also $\leq 3R^{-1}$ by (1) in Lemma \ref{lemma:q-norm<=2-norm}.
	
	By definition, we have
	$$|\psi_{v,R}(w)-\psi_{u,R}(w)|=|\psi_{0,R}(w-v)-\psi_{0,R}(w-u)|\leq \frac 3 R\|v-u\|_q.$$
	It follows that $\|\psi_{v,R}-\psi_{u,R}\|\leq 3R^{-1}\|v-u\|_q$.
\end{proof}
We also need to estimate the commutators between $F_{s,v}$ and functions.
\begin{lemma}\label{lemma:commutator}
	If $h$ be a smooth function on $E$ with $\ell^2$-Lipschitz constant $L_h$, then for any $s\geq 1$ and $v\in E$, $[F_{s,v},h]$ is compact and
	$$\|[F_{s,v},h]\|\leq\frac{9L_h}{4s}.$$
\end{lemma}
\begin{proof}
	For simplicity, we assume that $v=0$ and omit the subscript $v$. 
	Note that
	$$F_s=\frac 2\pi\int_0^\infty B_s\left(1+\lambda^2+B_s^2\right)^{-1}d\lambda,$$
	where the integral converges in the strong operator topology. For any $\lambda\geq 0$, we have
	\begin{align*}
		&[B_s\left(1+\lambda^2+B_s^2\right)^{-1},h]\\
		=&B_s[\left(1+\lambda^2+B_s^2\right)^{-1},h]+[B_s,h]\left(1+\lambda^2+B_s^2\right)^{-1}\\
		=&-B_s\left(1+\lambda^2+B_s^2\right)^{-1}\left(B_s[B_s,h]+[B_s,h]B_s\right)\left(1+\lambda^2+B_s^2\right)^{-1}\\
		&+[B_s,h]\left(1+\lambda^2+B_s^2\right)^{-1}
	\end{align*}
	Since $B_s=s^{-1}D+C$, we have 
	$$\|[B_s,h]\|= \|s^{-1}[D,h]\|\leq\frac{L_h}{s}.$$
	Note that for any $\lambda\geq 0$, 
	$$\|\left(1+\lambda^2+B_s^2\right)^{-1}\|\leq \frac{1}{1+\lambda^2},$$
	$$\|B_s\left(1+\lambda^2+B_s^2\right)^{-1}\|\leq
	\frac{1}{2\sqrt{1+\lambda^2}},
	$$
	$$\|B_s^2\left(1+\lambda^2+B_s^2\right)^{-1}\|\leq 1.
	$$
	Therefore,
	$$\|[B_s\left(1+\lambda^2+B_s^2\right)^{-1},h_R]\|\leq \frac{L_h}{s}\cdot \frac 9 4\cdot \frac{1}{1+\lambda^2}.$$
	Consequently, we have
	$$\|[F_s,h_R]\|\leq \frac 2 \pi\int_0^\infty\frac{L_h}{s}\cdot\frac 9 4\cdot\frac{1}{1+\lambda^2}d\lambda= \frac{9L_h}{4s}.$$
	This finishes the proof.
\end{proof}

We summarize the computation of this subsection in  the following proposition. 
\begin{proposition}\label{prop:phi}
	For any $\varepsilon\in(0,1)$, there is a smooth odd function $\Phi\colon\R\to\R$ such that 
	$$\|\Phi-x(1+x^2)^{-1/2}\|\leq \varepsilon,~  \lim_{x\to \pm \infty} \Phi(x)=\pm 1,$$
	$\Phi$ has compactly supported distributional Fourier transform,
	  and the operator 
	$$\Phi_{s,v}\coloneqq \Phi(B_{s,v})$$
	satisfies the following properties.
	\begin{enumerate}[label=$({\arabic*})$]
		\item For any $s\geq 1$ and $v\in E$, we have that $\|\Phi_{s,v}\|\leq 1$, and $(\Phi_{s,v})^2-1$ is compact.
		\item For each pair of $v,u\in E$, the path $\{\Phi_{s,v}\}_{s\in [1,+\infty)}$ is continuous with respect to strong operator topology, and the paths $\{(\Phi_{s,v})^2-1\}_{s\in [1,+\infty)}$ and $\{\Phi_{s,v}-\Phi_{s,u}\}_{s\in [1,+\infty)}$ are uniformly norm-continuous.
		\item There is $C_1>0$ such that for any function $h$ on $E$ with $\ell^2$-Lipschitz constant $L_h$, we have $\|[\Phi_{s,v},h]\|\leq s^{-1}C_1L_h$ for any $s\geq 1$.
		\item  $\Phi_{s,v}-\Phi_{s,u}$ is compact for all $v,u\in E$, and there is $L_1>0$ such that $\|\Phi_{s,v}-\Phi_{s,u}\|\leq L_1\|v-u\|_q$ for all $s\geq 1$.
		\item For each $r>0$, there exists $R_1>0$ such that $\|(1-\chi_{v,R_1})(\Phi_{s,v}-\Phi_{s,u})\|\leq 3\varepsilon$ for all $s\geq 2N^2L$ and  $v,u\in E$ with $\|v-u\|_q\leq r$, where $\chi_{v,R_1}$ is the characteristic function of $B_q(v,R_1)$.
		\item There is $R_2>0$ such that $\|(1-\chi_{v,R_2})((\Phi_{s,v})^2-1)\|\leq 3\varepsilon$ for any $v\in E$ and $s\geq 2N^2L$.
		\item\label{item:finiteProp} There is $C_2>0$ such that the $\ell^2$-propagation of $\Phi_{s,v}$ is at most $C_2s^{-1}$. Hence the $\ell^q$-propagation of $\Phi_{s,v}$ is also at most $C_2s^{-1}$.
	\end{enumerate}
\end{proposition}

In particular, the last item \ref{item:finiteProp} follows from the finite propagation speed of wave operators and that $\Phi$ has compactly supported distributional Fourier transform, and Lemma \ref{lemma:q-norm<=2-norm}.
\subsection{Twisted Roe algebras}\label{sec:twistedRoe}
In this subsection, we introduce the twisted Roe algebras and the associated $K$-theoretic maps. 

A general principle for proving the coarse Baum--Cones conjecture is to reduce the proof to a finite dimensional case. For example, in the case of the coarse Baum--Cones conjecture for spaces that coarsely embed into Hilbert space, the proof can be reduced to  the case of coarse disjoint unions of finite metric spaces that uniformly coarsely embed into Hilbert space \cite[Section 12.5]{willett2020higher}. 

 Let $X$ be a coarse disjoint union space, namely $X=\bigsqcup	 X_n$ with $d(X_n,X_m)\to \infty.$ Assume that $X$ coarsely embeds into $\ell^q$, and each $X_n$ is a metric space with finite diameter. Without loss of generality, we assume that $X$ is locally finite with bounded geometry. In this case, each $X_n$ is finite. We may assume that each $X_n$ coarsely embeds into a finite dimensional subspace $(E_n,\|\cdot\|_q)$. We will use $E$ to denote the collection of all $E_n$'s.
 
 Let us fix some notational conventions from \cite[Section 12.2]{willett2020higher}. Let $P_r(X)$ be the Rips complex of $X$ at scale $r$. Let $(B_{r,x})_{r\geq 0,x\in X}$ be the \emph{good cover system} of $P_r(X)$ for $X$ that satisfies the following properties:
 \begin{enumerate}
 	\item $(B_{r,x})_{x\in X}$ covers $P_r(X)$ by disjoint Borel sets;
 	\item $x\in B_{r,x}$, and $B_{r,x}$ is contained in the union of the simplices that contain $x$;
 	\item $B_{r,x}\subset B_{s,x}$ if $r\leq s$.
 \end{enumerate}
 See \cite[Definition 12.2.1]{willett2020higher}. Then the coarse embedding $f\colon X\to E$ extends to a Borel map $f\colon P_r(X)\to E$ by sending $B_{r,x}$ to $f(x)$, which is compatible with the inclusion $P_r(X)\subset P_s(X)$.
 
 Let $Z_r(X)$ be the collection of points in the simplicial complex $P_r(X)$ with rational coordinates. We fix $H_{P_r(X)}\coloneqq \ell^2(Z_r(X))\otimes\ell^2(\N)$ as a geometric $P_r(X)$-module. Elements in $B(H_{P_r(X)})$ are then represented as $X$-by-$X$ matrices: for each $T\in B(H_{P_r(X)})$, we set 
 $$T(x,y)=(\chi_{B_{r,x}}\otimes 1) T(\chi_{B_{r,y}}\otimes 1)\colon (\chi_{B_{r,y}}\otimes 1)H_{P_r(X)}\to (\chi_{B_{r,x}}\otimes 1)H_{P_r(X)}.$$
 It is clear that $T$ is locally compact if and only if $T(x,y)$ is compact for any $x,y\in X$, and $T$ has finite propagation if and only if $T(x,y)=0$ if $d(x,y)\geq d$ for some $d>0$. Therefore, since $X$ is a coarse disjoint union, any element $T$ in $C^*(P_r(X))$ associated to the module $H_{P_r(X)}$ is block-wise diagonal, that is, it is represented by $(T^n)_{n\in\N}$, where $T^n$ acts on $H_{P_r(X_n)}$. The following discussion will focus on $P_r(X)$ for some fixed $r$, hence we write $P=P_r(X)$ and $P_n=P_r(X_n)$ for simplicity. The elements in the good cover system will also be denoted by $(B_x)_{x\in X}$.

Let $H_{E_n}\coloneqq L^2(E_n,\Bigwedge^*E_n)$ be the Hilbert space of differential forms over $E_n$ with respect to the Euclidean metric $\|\cdot\|_2$ on $E_n$. We denote $H_{P,E}=\oplus(H_{P_r(X_n)}\otimes H_{E_n})$, which is both a geometric $X$-module and $E$-module. In fact, we will only discuss block-wise diagonal operators acting on $H_{P,E}$.

We define the twisted Roe algebra as follows.
\begin{definition}\label{def:Aalgebra}
	We define $A_{alg}(P,E)$ to be the collection of uniformly continuous paths of operators $\left(T_s=(T_s^n)_{n\in\N}\right)_{s\geq 1}$, where $T_s^n$ acts on $H_{P_n}\otimes H_{E_n}$, such that $T_s(x,y)$ is compact for any $x,y\in X$ and furthermore the following properties hold.
	\begin{enumerate}
		\item $\|T_s(x,y)\|$ is uniformly bounded on $X\times X\times[1,+\infty)$.
		\item $T_s$ has uniformly finite propagation along $P$. Equivalently, there exists $r_1>0$ such that $T_s(x,y)=0$ for all $s\geq 1$, provided that  $d(x,y)>r_1$. 
		\item the $\ell^q$-propagation of $T_s(x,y)$ is uniformly finite and tends to zero uniformly as $s\to\infty$ for all $x, y\in X$,  and
		\item \label{Aalgebra5} there exists $r_2>0$ such that the $E$-support of $T_s(x,y)$ is contained in the $\ell^q$-ball centered at $f(y)$ with radius $r_2$ for all $x,y\in X$.
	\end{enumerate} 
	The norm of $(T_s)$ is defined by $\sup_{s\geq 1}\|T_s\|$, and the the completion of $A_{alg}(P,E)$ is denoted by $A(X,E)$.
\end{definition}

Similarly, we define $A_L(P,E)$ to be the completion of the algebra of norm-bounded and norm-continuous maps $\mathbb T \colon [1,\infty) \to A_{alg}(P,E)$ such that   the propagation  of $\mathbb T(t)$ along $P$ tends to zero, as $t\to \infty$, where the completion is taken with respect to the norm
\[ \|\mathbb T\| =  \sup_{t\geq 1} \|\mathbb T(t)\|.\]

Let $C^*(P,E)$ be the Roe algebra associated with the geometric $P$-module $H_{P,E}$. Clearly, there is a homomorphism
$$\imath\colon A(P,E)\to C^*(P,E)$$
induced by evaluation at $s=1$ and  forgetting the last condition \eqref{Aalgebra5} in Definition \ref{def:Aalgebra}. 

There exists naturally an isomorphism
$$\id_*\colon K_*(C^*(P))\to K_*(C^*(P,E))$$
induced by tensoring with an arbitrary fixed rank-one projection. We shall show in this subsection that this isomorphism factors through $K_*(A(P,E))$, and in the next subsection the coarse Baum--Connes conjecture holds for $A(P,E)$.

Let $T\in B(H_P)$ and $S\in B(H_E)$ be block-wise diagonal operators. They induce operators $T\otimes 1$ and $1\otimes S$ that act on $H_{P,E}$. When no confusion is likely to arise,  we shall write $T$ for  $T\otimes 1$ and $S$ for $1\otimes S$. In the following, we will first construct an index map
$$\ind\colon K_*(C^*(P))\to K_*(A(P,E)).$$
We will only give the construction of the map $\ind$ for $K_0$ groups, since  the case of $K_1$ groups follows easily by a standard  suspension argument. 

Let $\Phi_{s,v}$ be the operator obtained from Proposition \ref{prop:phi}. Set 
\begin{equation}\label{eq:constant}
s_n\coloneqq 2L(\dim E_n)^2=2(1+q\cdot 2^{q})(\dim E_n)^2.
\end{equation}
Consider the  path of block-wise diagonal operators $\Phi = (\Phi_s)_{s\geq 1}$ with 
$$\Phi_s\colon H_{P,E}\to H_{P,E}$$
given by the formula
$$\Phi_s\xi=\Phi_{s+s_n,f(x)}(\xi) \textup{ for all } x\in X_n \text{ and }\xi\in (\chi_{B_x}\otimes 1)H_{P,E}, $$
where $f\colon X\to \ell^q$ is the given coarse embedding. With respect to the even-odd grading on each $\Bigwedge^*E_n$, we have
$$\Phi=\begin{pmatrix}
	0&\Phi_-\\
	\Phi_+&0
\end{pmatrix}$$
as each $\Phi_{s,v}$ is obtained by  the functional calculus of an odd function of  an odd-graded operator.

Let $\psi_{v,R}$ be the cut-off function  from Lemma \ref{lemma:cut-off}. For each $R>0$, we define a bounded operator 
$$\psi_R\colon H_{P,E}\to H_{P,E} $$
by 
$$\psi_R\xi=\psi_{f(x),R}\xi\textup{ for all } \xi\in (\chi_{B_x}\otimes 1)H_{P,E}.$$

Let $\alpha$ be an element of $K_0( C^*(P))$. Then $\alpha=[\cP] - [\cP']$, where $\cP$ and $\cP'$ are projections in $M_n(C^*(P)^+)$ such that  $\cP-\cP'\in M_n(C^*(P))$. In the following, we assume that $\cP\in C^*(P)$ without loss of generality, and explain  the construction of the index map on the element $\cP$.

Let  $\cP\otimes 1\colon H_{P,E} \to H_{P,E}$ be the operator constructed from the operator $\cP$. For simplicity, we shall write $\cP$ in place of $\cP\otimes 1$ below.
Let us consider the following idempotent.
\begin{equation}\label{eq:index}
	\cP_\Phi\coloneqq \begin{pmatrix}
		\cP-(\cP-\cP\Phi_+\cP\Phi_-)^2&(2\cP-\cP\Phi_+\cP\Phi_-)\cP\Phi_+(\cP-\cP\Phi_-\cP\Phi_+)\\ \cP\Phi_-(\cP-\cP\Phi_+\cP\Phi_-)&(\cP-\cP\Phi_-\cP\Phi_+)^2
	\end{pmatrix}
\end{equation}
where the above expression follows the standard formula for  the index class of the generalized Fredholm operator $\cP\otimes \Phi$.  For $R>0$, we define
\begin{align*}
	\cP_{\Phi,R}\coloneqq &\begin{pmatrix}
		\cP-\psi_R(\cP-\cP\Phi_+\cP\Phi_-)^2\psi_R&\psi_R(2\cP-\cP\Phi_+\cP\Phi_-)\cP\Phi_+(\cP-\cP\Phi_-\cP\Phi_+)\psi_R\\\psi_R\cP\Phi_-(\cP-\cP\Phi_+\cP\Phi_-)\psi_R&\psi_R(\cP-\cP\Phi_-\cP\Phi_+)^2\psi_R
	\end{pmatrix}.
\end{align*}
The following proposition shows that the map 
\[  [\cP] \mapsto [\cP_{\Phi,R}]-[\begin{psmallmatrix}
	\cP & 0 \\ 0 & 0 
\end{psmallmatrix}]\]
gives a well-defined homomorphism
 $$\ind\colon K_*(C^*(P))\to K_*(A(P,E)).$$
\begin{proposition}\label{prop:indMap}
	If $\cP\in C^*(P)$ and $\cP_{\Phi,R}$ are as above, then
	$$\cP_{\Phi,R}-\begin{psmallmatrix}
		\cP & 0 \\ 0 & 0 
	\end{psmallmatrix}\in A(P,E)$$
	and there is $C>0$ such that $\|\cP_{\Phi,R}\|\leq C$. Furthermore, if $\cP\in C^*(P)$ and $\cP^2=\cP$, then there exist $R>0$ and a function $\Phi$ as in Proposition \ref{prop:phi} depending only on $\cP$ such that
	$$\|(\cP_{\Phi,R})^2-\cP_{\Phi,R}\|< \frac 1 4.$$
	Hence the formal difference $[P_{\Phi,R}]-[\begin{psmallmatrix}
		P & 0 \\ 0 & 0 
	\end{psmallmatrix}]$ defines a class in $K_0(A(P,E))$, denoted by $\ind([\cP])$.
\end{proposition}
\begin{proof}
	The norm estimates for $\cP_{\Phi,R}$ and $\cP_{\Phi}$ are clear from construction and part (1) of Proposition \ref{prop:phi}.	
	Let $\cP\in C^*(P)$. 
	We first note that $\psi_R\cP(\Phi^2-1)\psi_R\in A(P,E)$ by parts (1), (2), (4) and (7) of Proposition \ref{prop:phi}, and $\psi_R[\cP,\Phi]\psi_R\in A(P,E)$ by parts  (2), (4) and (7) of Proposition \ref{prop:phi}. Therefore, we have
	$$\cP\Phi \cP\Phi-\cP=\cP(\Phi^2-1)-\cP[\cP,\Phi]\Phi\in A(P,E).$$
	Note that
	$$\cP\Phi \cP\Phi-\cP=\begin{pmatrix}
		\cP\Phi_-\cP\Phi_+-\cP&\\&\cP\Phi_+ \cP\Phi_--\cP
	\end{pmatrix},$$
	and every entry of $\cP_{\Phi,R}-\begin{psmallmatrix}
		\cP & 0 \\ 0 & 0 
	\end{psmallmatrix}$ is composed of operators that are either $\cP\Phi_- \cP\Phi_+-\cP$  or $\cP\Phi_+ P\Phi_--\cP$. It follows that  $\cP_{\Phi,R}-\begin{psmallmatrix}
	\cP & 0 \\ 0 & 0 
	\end{psmallmatrix}\in A(P,E)$.
	
	Since $C^*_{alg}(P)$ is dense in $C^*(P)$, for any $\delta>0$, there exists $\cQ\in C^*_{alg}(P)$ such that $\|\cQ\|=\|\cP\|$ and $\|\cP-\cQ\|<\delta$. In particular, we have 
	\begin{itemize}
		\item $\sup_{x,y\in X}|\cQ(x,y)|<\infty$,
		\item $\|\cQ^2-\cQ\|<\delta$.
		\item the propagation of $\cQ$ is finite, hence $\cQ(x,y)=0$ if $d(x,y)>r_1$ for some $r_1>0$,
	\end{itemize}
	To show that $\|(\cP_{\Phi,R})^2-\cP_{\Phi,R}\|< 1/4$, it suffices to estimate $\|\cP_{\Phi,R}-\cP_\Phi\|$, since $\cP_\Phi$ is an idempotent. In fact, since $\|\cQ^2-\cQ\|<\delta$, by choosing $\delta $ to be sufficiently small, it suffices to estimate $\|\cQ_{\Phi,R}-\cQ_\Phi\|$, where  
	\[  	\cQ_\Phi\coloneqq \begin{pmatrix}
	  \cQ-(\cQ-\cQ\Phi_+\cQ\Phi_-)^2&(2\cQ-\cQ\Phi_+\cQ\Phi_-)\cQ\Phi_+(\cQ-\cQ\Phi_-\cQ\Phi_+)\\ \cQ\Phi_-(\cQ-\cQ\Phi_+\cQ\Phi_-)&(\cQ-\cQ\Phi_-\cQ\Phi_+)^2
	\end{pmatrix} \]
	and $\cQ_{\Phi,R}$ is defined by the same formula  as $\cP_{\Phi, R}$ but with $\cP$ replaced by $\cQ$.

	 The matrix decomposition of $[\cQ,\psi_R]$ is given by
	$$\left([\cQ,\psi_R]\right)(x,y)=\cQ(x,y)(\psi_{f(x), R}-\psi_{f(y), R}).$$
	In particular, if  $\left([\cQ,\psi_R]\right)(x,y)\ne 0$, then 
	$$\|f(x)-f(y)\|_q\leq \rho_+(r_1)$$
	where $\rho_+$ is the function as given in Definition  \ref{def:coarseembed}. Therefore by Lemma \ref{lemma:cut-off}, we have
	$$\|\left([\cQ,\psi_R]\right)(x,y)\|\leq 3R^{-1}\rho_+(r_1).$$
	Therefore, we have
	$$[\cQ,\psi_R]\to 0\text{ as }R\to\infty.$$
	Furthermore, it follows from Lemma \ref{lemma:cut-off} and part (3) of Proposition \ref{prop:phi} that
	$$[\Phi,\psi_R]\to 0\text{ as }R\to\infty.$$
	Now to estimate $\| \cQ_{\Phi,R}-\cQ_\Phi\|$, we need to estimate the norms of terms such as 
	\[ (\cQ\Phi \cQ\Phi-\cQ)- \psi_R(\cQ\Phi \cQ\Phi-\cQ)\psi_R. \]
In the following, let us only give the details on how to estimate  $\|(\cQ\Phi \cQ\Phi-\cQ)- \psi_R(\cQ\Phi \cQ\Phi-\cQ)\psi_R\|$, since the other terms can be estimated in a completely similar way. Note that 
	\begin{align*}
	\cQ\Phi \cQ\Phi-\cQ &= - \cQ[\cQ, \Phi]\Phi - \cQ + \cQ^2\Phi^2 \\ & = - \cQ[\cQ, \Phi]\Phi + (\cQ^2- \cQ)\Phi^2 + \cQ(\Phi^2-1).
\end{align*}
It follows that 
	\begin{align*}
		&\|(\cQ\Phi \cQ\Phi-\cQ)-\psi_R(\cQ\Phi \cQ\Phi-\cQ)\psi_R\|\\
		\leq &\|\cQ(\Phi^2-1) -\psi_R\cQ(\Phi^2-1)\psi_R\| +\|\cQ[\cQ,\Phi]\Phi-\psi_R\cQ[\cQ,\Phi]\Phi\psi_R\|\\
		&\quad +\|(\cQ^2-\cQ)\Phi^2\|+\|\psi_R(\cQ^2-\cQ)\Phi^2\psi_R\|.
	\end{align*}
The last two terms of the right hand side  are both $<\delta$.  The first term of the right hand side satisfies 
\begin{align*}
	& \|\cQ(\Phi^2-1) -\psi_R\cQ(\Phi^2-1)\psi_R\| \\
	\leq  & \|\cQ(\Phi^2-1) - \cQ \psi_R (\Phi^2-1)\psi_R \|  + \|[\psi_R,\cQ] (\Phi^2-1)\psi_R\|  \\
	\leq & \|\cQ (1-\psi_R) (\Phi^2-1) \|  +  \|\cQ \psi_R (\Phi^2-1) (1-\psi_R) \| +  \|[\psi_R,\cQ] (\Phi^2-1)\psi_R\|
\end{align*}
It follows that part  (6) of Proposition \ref{prop:phi} that $\|\cQ(\Phi^2-1) -\psi_R\cQ(\Phi^2-1)\psi_R\|$ can be made arbitrarily small as long as  the number $\varepsilon$ from Proposition \ref{prop:phi} is chosen to be sufficiently small and $R$ is chosen to be sufficiently large. Similarly, the term $\|\cQ[\cQ,\Phi]\Phi-\psi_R\cQ[\cQ,\Phi]\Phi\psi_R\|$ satisfies 
\begin{align*}
	& \|\cQ[\cQ,\Phi]\Phi-\psi_R\cQ[\cQ,\Phi]\Phi\psi_R\| \\
	\leq & \|\cQ[\cQ,\Phi]\Phi-\cQ\psi_R[\cQ,\Phi]\psi_R \Phi \| + \|\cQ\psi_R[\cQ,\Phi][\Phi, \psi_R]\| \\
& 	\quad  + \|[\psi_R,\cQ] [\cQ,\Phi]\Phi\psi_R \|.
\end{align*}
It follows that   that it suffices for us to  estimate the norm
\[  \|[\cQ,\Phi] - \psi_R[\cQ,\Phi]\psi_R\| \] 
which in turn can be estimated as follows: 
\begin{align*}
	& \|[\cQ,\Phi] - \psi_R[\cQ,\Phi]\psi_R\| \\
	\leq & \|(1-\psi_R)[\cQ,\Phi]\|  + \| \psi_R[\cQ,\Phi](1-\psi_R) \|.
\end{align*}
Note that the matrix decomposition of  $(1-\psi_R)[\cQ,\Phi]$ is given by 
$$\left( (1-\psi_R)[\cQ,\Phi]\right) (x, y)=(1-\psi_{f(x), R})(\Phi_{f(x)}-\Phi_{f(y)}) \cQ(x,y),$$
Since $\cQ$ has finite propagation, it follows from part   (5)  of Proposition \ref{prop:phi} that  $\|(1-\psi_R)[\cQ,\Phi]\|$ can be made arbitrarily small as long as  the number $\varepsilon$ from Proposition \ref{prop:phi} is chosen to be sufficiently small and $R$ is chosen to be sufficiently large. The same estimate also applies to $\| \psi_R[\cQ,\Phi](1-\psi_R) \|$. 

To summarize, we have shown that  the norm  $\| \cQ_{\Phi,R}-\cQ_\Phi\|$ can be made arbitrarily small, as long as  the number $\varepsilon$ from Proposition \ref{prop:phi} and the number $\delta$ are  chosen to be sufficiently small and $R$ is chosen to be sufficiently large. Therefore,  as long as  the numbers $\varepsilon,\delta,R$ are chosen, $\|(\cP_{\Phi,R})^2-\cP_{\Phi,R}\|$ can be made arbitrarily small, thus  in particular less than $1/4$.  This finishes the proof. 
\end{proof}

\begin{proposition}\label{prop:composition}
	The composition
	$$\imath_*\circ\ind\colon K_*(C^*(P))\to K_*(A(P,E))\to K_*(C^*(P,E))$$
	is equal to the isomorphism $\id_*\colon K_*(C^*(P))\to  K_*(C^*(P,E))$ induced by the tensor product with a rank-one projection.
\end{proposition}
\begin{proof}
Define a map $\kappa\colon (E, \|\cdot\|_q) \to (E, \|\cdot\|_q)$ by
	$$\kappa(x)=\begin{cases}
		\frac{x}{\|x\|_q}(\|x\|_q-1)& \textup{ if } \|x\|_q\geq 1,\\
		0 &\textup{ if } \|x\|_q<1.
	\end{cases}$$
Clearly $\kappa$ is $1$-Lipschitz and $\kappa^m(x)\coloneqq \kappa(m^{-1}x) = 0$ if $m\geq\|x\|_q$. Furthermore,  we set  $\kappa^0(x) = x$ when $m=0$,  and similarly $\kappa^\infty(x) \equiv 0$ when $m=\infty$. 
Recall that $\Phi$ is the path of operators 
$\Phi_s\colon H_{P,E}\to H_{P,E}$
given by 
$$\Phi_s(\xi)=\Phi_{s+s_n,f(x)}(\xi)\text{ for all }\xi\in (\chi_{B_x}\otimes 1)H_{P,E}$$
for $s\in[1,+\infty)$, where $f\colon X\to \ell^q$ is the given coarse embedding and $s_n$ is the constant given in line \eqref{eq:constant}.
We define $\Phi^m_1\colon H_{P,E}\to H_{P,E}$ by
$$\Phi^m_1(\xi)=\Phi_{1+s_n,\kappa^m(f(x))}(\xi)\text{ for all }\xi\in (\chi_{B_x}\otimes 1)H_{P,E}.$$
Note that,  when $m=\infty$, $\kappa^m(f(x))$ equals the origin for all $x\in X$. In particular, the base point for the operator $\Phi_{1+s_n,\kappa^\infty(f(x))}$ is the origin of $E_n$ for all $x\in X_n$. 

Consider $[\cP]\in K_0(C^*(P))$ where $\cP\in C^*(P)$ and $\cP^2=\cP$. Similar to the proof of Proposition \ref{prop:indMap}, 
for  any $\delta>0$, let $\cQ\in C^*_{alg}(P)$ be an almost idempotent  such that $\|\cQ\|=\|\cP\|$ and $\|\cP-\cQ\|<\delta$.
 Define  $\cQ_{\Phi^m_1}$ similar to the formula in line \eqref{eq:index}. By the proof of Proposition \ref{prop:indMap}, we have
$$\cQ_{\Phi^m_1}- \begin{psmallmatrix}
	\cQ & 0 \\ 0 & 0 
\end{psmallmatrix}\in A(P,E),\text{ and }\|(\cQ_{\Phi^m_1})^2-\cQ_{\Phi^m_1}\|<1/4,$$
for all $m\in \mathbb N\cup \{\infty\}$. Here we remark that we have used the same function $\Phi$ (as in Proposition \ref{prop:phi}) to perform the functional calculus for all  $m\in \mathbb N\cup \{\infty\}$.

Now we define 
$$a\coloneqq \left[\begin{pmatrix}
	\cQ_{\Phi^0_1}&&&\\
	&\cQ_{\Phi^1_1}&&\\
	&&\cQ_{\Phi^2_1}&\\
	&&&\ddots
\end{pmatrix}\right]-\left[\begin{pmatrix}
\cQ_{\Phi^\infty_1}&&&\\
&\cQ_{\Phi^\infty_1}&&\\
&&\cQ_{\Phi^\infty_1}&\\
&&&\ddots
\end{pmatrix}\right]$$
and 
$$b\coloneqq [\cQ_{\Phi^0_1}]-[\cQ_{\Phi^\infty_1}].$$
Both $a$ and $b$ are formal differences of almost idempotents as long as  $\delta$ is sufficiently small. It is clear that $b$ defines a class in $ K_0(C^*(P,E))$. As for $a$, we observe that,  for any fixed  $y\in X$,  
\[ \cQ_{\Phi^m_1}(x,y)=\cQ_{\Phi^\infty_1}(x,y) \textup { for $\forall x\in X$ and all but finitely many $m$'s,} \] 
since $\cQ$ has finite propagation.  Thus for any $(x,y)$, the difference
$$\begin{pmatrix}
	\cQ_{\Phi^0_1}&&&\\
	&\cQ_{\Phi^1_1}&&\\
	&&\cQ_{\Phi^2_1}&\\
	&&&\ddots
\end{pmatrix}(x,y)-\begin{pmatrix}
	\cQ_{\Phi^\infty_1}&&&\\
	&\cQ_{\Phi^\infty_1}&&\\
	&&\cQ_{\Phi^\infty_1}&\\
	&&&\ddots
\end{pmatrix}(x,y)$$
is a finite direct sum of compact operators, hence still compact.

 Therefore, by identifying $H_E$ with a direct sum of countably many copies of $H_E$, we see that $a$ defines a class in $K_0(C^*(P,E))$. Now consider the obvious homotopy of almost idempotents that connects $\cQ_{\Phi^m_1}$ to $\cQ_{\Phi^{m+1}_1}$. Together they produce a homotopy of almost idempotents between $a+b$ and $a$, which implies that  $b=0$. In other words, 
 \[[\cQ_{\Phi^0_1}]=[\cQ_{\Phi^\infty_1}].   \]
 Since the base point for the operator $\Phi_{1+s_n,\kappa^\infty(f(x))}$ is the origin of $E_n$ for all $x\in X_n$, it follows that  $\Phi^\infty_1$ commutes $Q$. Therefore, the formula for $\cQ_{\Phi^\infty_1}$ becomes 
 \[ \cQ_{\Phi^\infty_1} = \cQ\otimes \ind (\Phi^\infty_1)\]
 whose $K$-theory class is equal to that of $\cQ$ tensor with a rank-one projection, according to  Lemma \ref{lemma:Fredholm}. To summarize, we have shown  that the composition $i_*\circ\ind$ coincides with the map induced by the tensor product on each $X_n$ by a rank-one projection, which is clearly the identity map at the level of $K$-theory.
\end{proof}
\subsection{Coarse Baum--Connes conjecture for twisted algebras}\label{sec:cbcTwist}
The work in the previous subsection allows us to reduce our proof of the coarse Baum--Connnes conjecture (cf. Conjecture \ref{conj:cBC}) for spaces that are coarsely embeddable into $\ell^q$ (with $q\in [1, \infty)$) to the following theorem, which should be viewed as a coarse Baum--Connnes conjecture with certain coefficients. 
\begin{proposition}\label{prop:cbcTwisted}
	The map 
	$$\ev_*\colon \lim_{r\to\infty}K_*(A_L(P_r(X),E))\to K_*(A(P_r(X),E))$$
	is an isomorphism.
\end{proposition}

We shall prove the above proposition in several steps. First let us  prove the following elementary lemma. 
\begin{lemma}\label{lemma:bddGeometry}
	Let $X$ be a locally finite metric space with bounded geometry, and $f\colon X\to \ell^q$ a coarse embedding. Then for any $R>0$, the family of balls $\{B_q(f(x),R)\}_{x\in X}$ in $\ell^q$ has finite multiplicity.
\end{lemma}
\begin{proof}
	Assume that $B_q(f(x_1),R),B_q(f(x_2),R),\cdots,B_q(f(x_n),R)$ have non-trivial intersection, that is,
	$$v\in \bigcap_{i=1}^nB_q(f(x_i),R).$$
	It follows that
	$$\|f(x_i)-f(x_j)\|\leq \|f(x_i)-v\|+\|f(x_j)-v\|\leq 2R,~\forall i,j\leq n.$$
	Since the map $f$ is a coarse embedding, there exists $R'>0$ such that
	$$d_X(x_i,x_j)\leq R',~\forall i,j\leq n,$$
	that is, the set $\{x_1,\ldots,x_n\}$ has diameter no more than $R'$. Since $X$ has bounded geometry, there exists $N_{R'}>0$ such that  any set in $X$ with diameter $\leq R'$ has cardinality $\leq N_{R'}$. Therefore $n\leq N_{R'}$, that is, the family  $\{B_q(f(x),R)\}_{x\in X}$ has multiplicity no more than $N_{R'}$. This finishes the proof. 
\end{proof}

Now we construct some ideals of the twisted algebras that restricts elements to certain subsets. We shall analyze the assembly map locally on these ideas.
\begin{definition}\label{def:AalgebraFinite}
	Given the Rips complex $P$, for any $R>0$, we define $A_{alg}(P,E)^R$ as a subalgebra of $A_{alg}(P,E)$ consisting of  paths of operators $T_s$ such that for any $\varepsilon>0$, there exists $s_\varepsilon>0$ such that  the $E$-support of $T_s(x,y)$ is contained in the $(R+\varepsilon)$-neighborhood of $\{f(x),f(y)\}$ (with respect to the $\ell^q$-metric) for all $x,y\in X$ and all $s>s_\varepsilon$.
Denote by $A(P,E)^R$ the completion of $A_{alg}(P,E)^R$ in $A^p(P,E)$.
\end{definition}
For example, the operator $\mathcal Q_\Phi$ constructed in the proof of Proposition \ref{prop:indMap} lies in $A_{alg}(P,E)^R$ for some $R>0$.
\begin{definition}\label{def:AalgebraLocal}
	Let $U$ be a subset of $\ell^q$. We define $A_{alg}(P,E)_U^R$ as a subalgebra of $A_{alg}(P,E)^R$ consisting of paths of  operators $T_s$ such that for any $\varepsilon>0$ there exists $s_\varepsilon>0$ such that the $E$-support of $T_s(x,y)$ is contained in the $\varepsilon$-neighborhood of $U$ (with respect to the $\ell^q$-metric) for all  $x,y\in X$  and all $s>s_\varepsilon$. We denote by  $A(P,E)_U^R$  the completion of $A_{alg}(P,E)_U^R$ in $A(P,E)^R$.
\end{definition}

Recall that by definition, $A_{alg}(P,E)$ consists of paths of operators whose $\ell^q$-propagation along $E$ goes to zero. As a result, $A_{alg}(P,E)^R$ and $A_{alg}(P,E)_U^R$ are closed two-sided ideals of $A_{alg}(P,E)$ for any $R>0$ and $U\subset\ell^q$. Furthermore, we have
$$A_{alg}(P,E)=\lim_{R\to\infty}A_{alg}(P,E)^R.$$

As the first step for the proof of Proposition \ref{prop:cbcTwisted}, we prove that if $X$ is coarsely embeddable into $\ell^q$, then   the coarse Baum--Connes map for $A(P,E)_U^R$ is  an isomorphism for each fixed $R>0$ and each  bounded set $U$. More precisely, we have the following lemma. 
\begin{lemma}\label{lemma:cbcBounded}
	If $U\subset \ell^p$ is a bounded set, then 
		$$\ev_*\colon \lim_{r\to\infty}K_*(A_L(P_r(X),E)^R_U)\to K_*(A(P_1(X),E)^R_U)$$
	is an isomorphism for any $R>0$.
\end{lemma}
\begin{proof}
	Let  $\cP\in A(P_1(X),E)_U^R$  be an idempotent. For  any $\delta>0$, let $\cQ = (\cQ_s)_{s\geq 1}\in A_{alg}(P_1(X),E)_U^R$ be an almost idempotent, namely $\|\cQ\|=\|\cP\|$,  $\|\cP-\cQ\|<\delta<1/10$, and $\cQ_s$ has uniform finite propagation along $P_1(X)$ for all $s\geq 1$. 
	
	According to Definition \ref{def:AalgebraFinite}, we may start  the path $(\cQ_s)_{s\geq 1}$ at $s=s_\varepsilon$ and reparameterize the path so that the $E$-support of $\cQ_s(x,y)$ is contained in $B_q(f(y),R+1)\cup B_q(f(x),R+1)$
	for all $x,y\in X$ and $s\geq 1$. 
   Similarly, according to  Definition \ref{def:AalgebraLocal}, we may assume furthermore that the $E$-support of $\cQ_s(x,y)$ is also contained in $N_q(U,1)$, where $N_q(U,1)$ is the $1$-neighborhood of $U$ with respect to the $\ell^q$-metric. As a result, the $E$-support of $\cQ_s(x,y)$ is contained in
	$$N_q(\{f(x),f(y)\},R+1)\cap N_q(U,1).$$
	Therefore, $\cQ_s(x,y)$ is non-zero only if either  $f(x)$ or $f(y)$ is in $N_{q}(U,R+2)$. 
	As $\cQ$ has finite propagation along $P_1(X)$, there eixsts $r_1>0$ such that $\cQ_s(x,y)=0$ for any $s\geq 1$ and $x,y\in X$ with $d(x,y)\geq r_1$. Therefore, since $f$ is a coarse embedding, $\cQ_s(x,y)$ is non-zero for some $s\geq 1$ only if both $x$ and $y$ lie in a fixed bounded subset of $X$, which will be denoted by $Y$.
	
	 By construction, the set $Y$ lies in a single simplex of the Rips complex $P_r(X)$ as long as $r$ is sufficiently large. This reduces the proof to the case when  $X$ is a single point $\{\star\}$, and  it is clear that   
	 $$\ev_*\colon \lim_{r\to\infty}K_*(A_L(P_r(\{\star\}),E)^R_U)\to K_*(A(P_1(\{\star\}),E)^R_U)$$
	 is an isomorphism in this case. See for example  \cite[Theorem 5.1.15 \& Theorem 6.4.16]{willett2020higher}. This finishes the proof. 
\end{proof}
\begin{corollary}\label{coro:cbcBounded}
		If $R>0$ and $U=\bigsqcup_{i\geq 1} U_i$, where the $U_i$'s are uniformly bounded and $\delta$-disjoint from each other for some $\delta>0$, then 
$$\ev_*\colon \lim_{r\to\infty}K_*(A_L(P_r(X),E)^R_U)\to K_*(A(P_1(X),E)^R_U)$$
	is an isomorphism.
\end{corollary}
\begin{proof}
Note that  
	$$K_*(A(P_1(X),E)^R_U)=\prod_{i\geq 1} K_*(A(P_1(X),E)^R_{U_i}).$$
	Now the conclusion follows directly from Lemma \ref{lemma:cbcBounded}.
\end{proof}

Now by a standard Mayer--Vietoris argument, we have the following corollary.
\begin{corollary}\label{coro:cbcMV}
	If $R>0$ and $U=\bigcup_{i\geq 1} U_i$ such that  the $U_i$'s are uniformly bounded and the collection $\{U_i\}$ has  finite multiplicity, then 
	$$\ev_*\colon \lim_{r\to\infty}K_*(A_L(P_r(X),E)^R_U)\to K_*(A(P_1(X),E)^R_U)$$
	is an isomorphism.
\end{corollary}

Now we are ready to prove Proposition \ref{prop:cbcTwisted}.
\begin{proof}[Proof of Proposition \ref{prop:cbcTwisted}]
	Let
	$$U_R=\bigcup_{x\in X}B_q(f(x),R+1),$$
	where $B_q(f(x),R)$ is the ball of radius $R$ center at $f(x)$ in $\ell^q$. 
	It follows from Definitions \ref{def:AalgebraFinite} and \ref{def:AalgebraLocal} that 
	$$A(P_r(X),E)^R=A(P_r(X),E)_{U_R}^R$$
	for any scale $r>0$.
	
	By Lemma \ref{lemma:bddGeometry},  the collection $\{B_q(f(x),R+1)\}_{x\in X}$ has finite multiplicity in $X$. Thus by Corollary \ref{coro:cbcMV}, the map 
	$$\ev_*\colon \lim_{r\to\infty}K_*(A_L(P_r(X),E)^R_{U_R})\to K_*(A(P_1(X),E)^R_{U_R})$$
	 is an isomorphism for every $R>0$. Now the proof is finished by observing that 
	 \[ \lim_{R\to \infty} \lim_{r\to\infty}K_*(A_L(P_r(X),E)^R) = \lim_{r\to\infty}K_*(A_L(P_r(X),E))\] 
	 and 
	 \[\lim_{R\to \infty} K_*(A(P_r(X),E)^R) = K_*(A(P_1(X),E)).\]
\end{proof}

By combining the above discussion, now we  prove Theorem \ref{thm:lqcBC}.
\begin{proof}[Proof of Theorem \ref{thm:lqcBC}]
As discussed at the beginning of Section \ref{sec:twistedRoe}, to prove 
\[ \ev_*\colon \lim_{r\to\infty}K_*( C^*_L(P_r(X)))\to K_*( C^*(X))\]
is an isomorphism, it suffices to prove the case 	where  $X$ is a coarse disjoint union, cf.  \cite[Section 12.5]{willett2020higher}. 

It follows from Definition \ref{def:Aalgebra}, Definition \ref{def:Roe} and Proposition \ref{prop:indMap} that we have the following commutative diagram	\[
	\begin{tikzcd}\displaystyle
		\lim_{r\to\infty}K_*(C^*_L(P_r(X)) \arrow[rd, "\ind"] \arrow[ddd,"ev_*"] \arrow[rr, dashed] &                                                                                      & \displaystyle\lim_{r\to\infty}K_*(C^*_L(P_r(X),E)) \arrow[ddd,"ev_*"] \\
		& {\displaystyle\lim_{d\to\infty}K_*(A_L(P_r(X),E))} \arrow[ru, "\imath_*"] \arrow[d, "ev_*"] &                                                      \\
		& {K_*(A(P_1(X),E))} \arrow[rd,"\imath_*"]                                                         &                                                      \\
		K_*(C^*(X)) \arrow[ru,"\ind"] \arrow[rr, dashed]                                              &                                                                                      & {K_*(C^*(P_1(X),E))}                                   
	\end{tikzcd}
	\]
 The horizontal dashed arrows are isomorphisms by Proposition \ref{prop:composition}. The vertical map $\ev_*$ in the middle is an isomorphism by Proposition \ref{prop:cbcTwisted}. A standard diagram chasing shows that  the left vertical map $ev_*$ is an isomorphism. This finishes the proof.  
\end{proof}

\section{$\ell^p$-coarse Baum--Connes conjecture for coarsely embeddable spaces}\label{sec:lplqCBC}
In this section, we prove Theorem \ref{thm:lp-lqcBCIntro} for general $p\geq 1$.
\subsection{The $\ell^p$-coarse Baum--Connes conjecture}
 In this subsection, we introduce $\ell^p$-Roe algebras and prove the corresponding $\ell^p$-coarse Baum--Connes conjecture. We point out  that the $\ell^p$-Roe algebras in the current paper are slightly different from the version of $\ell^p$-Roe algebras introduced by Zhang and Zhou in  \cite{Zhangzhoulp}. See the discussion following Conjecture \ref{conj:l^pCBC} for more details. 

For simplicity, let $P$ be a finite dimensional simplicial complex with bounded geometry, e.g., the Rips complex of a locally finite metric space. Let $Z$ be a countable dense subset of $P$. Denote by $\N$ the set of natural numbers and $H=\ell^2(\N)$. We denote by $\ell^p(Z,\ell^p(\N,H))=\ell^p(Z\times \N,H)$ the completion of simple functions $\varphi\colon Z\times\N\to H$ with the following norm
$$\|\varphi\|_p\coloneqq \left(\sum_{(z,n)\in Z\times\N}\|\varphi(z,n)\|^p_H\right)^{1/p}.$$
Equivalently, the Banach space $\ell^p(Z,\ell^p(\N,H))$ can be viewed as  the space of functions on $Z\times\N\times\N$ with respect to the following mixed $\ell^p$-$\ell^2$ norm:
$$\|\varphi\|_p= \left(\sum_{(z,n)\in Z\times\N}\left(\sum_{m\in\N}|\varphi(z,n,m)|^2\right)^{p/2}\right)^{1/p}.$$
Each compactly supported bounded Borel function on $P$ acts on $\ell^p(Z,\ell^p(\N,H))$ by multiplication on the first factor $Z$. Analogous to the $\ell^2$ case,  we shall use the space $\ell^p(Z,\ell^p(\N,H))$ as the geometric  module to construct the corresponding $\ell^p$-Roe algebra. 
\begin{definition}\label{def:lpPropLcpt} Let $T$ be a bounded operator on $\ell^p(Z,\ell^p(\N,H))$.
	\begin{enumerate}
		\item The \emph{support} of $T$, denoted by $\supp(T)$, is defined to be the collection of points $(x,y)$ in $Z\times Z$, such that $\chi_VT\chi_U\ne 0$ for all open neighborhoods $U$ of $x$ and $V$ of $y$, where $\chi_U$ and $\chi_V$ are the characteristic functions of $U$ and $V$, respectively.
		\item The \emph{propagation} of $T$, denoted by $\prop(T)$, is defined to be the supremum of $d(x,y)$ for all $(x,y)\in\supp(T)$.
		\item $T$ is called \emph{locally compact}, if $\chi_KT$ and $T\chi_K$ are compact operators on $\ell^p(Z,\ell^p(\N,H))$ for any compact set $K\subset Z$.
	\end{enumerate}
\end{definition}
\begin{definition}\label{def:lpRoe}
	Let $P$ be a finite dimensional simplicial complex. Let $ C^p_{alg}(P)$ be the collection of all bounded operators on $\ell^p(Z,\ell^p(\N,H))$ that are locally compact and have finite propagation. The $\ell^p$-Roe algebra of $P$, denoted by $C^p(P)$, is defined to be the Banach algebra obtained by the completion of $ C^p_{alg}(P)$ with respect to the operator norm.
\end{definition}

The $K$-theory of the $\ell^p$-Roe algebras is invariant under coarse equivalences. This follows from the same proof of \cite[Lemma 2.8]{Zhangzhoulp}. Here we omit the details.
\begin{lemma}\label{lemma:coarseEq}
	Let $P$ and $P'$ be simplicial complexes and $f\colon P\to P'$ a continuous coarse map. Then $f$ cannonically induces a map
	$$f_*\colon K_*(C^p(P))\to K_*(C^p(P')).$$
	In particular, $C^p(P)$ is independent of the choice of the countable dense subset $Z$.
\end{lemma}

Now we formulate the $\ell^p$-coarse Baum--Connes conjecture.
\begin{definition}\label{def:lpLocalization}
		Let $P$ be a finite dimensional simplicial complex. Let $C_L^p(P)$ be the completion of all uniformly continuous paths
		$$T\colon [0,\infty)\to C^p_{alg}(P)$$
		such that $\|T(t)\|$ is uniformly bounded and $\prop(T(t))\to 0$ as $t\to\infty$, with respect to the sup-norm
		$$\|T\|\coloneqq \sup_{t\geq 0}\|T(t)\|.$$
\end{definition}

Recall that in Section \ref{sec:preliminary} we have introduced the Rips complex $P_r(X)$ at scale $r$ for a locally finite metric space $X$ with bounded geometry. Note that $P_d(X)$ is coarsely equivalent to $X$, hence the $K$-groups of their $\ell^p$-Roe algebras are canonically isomorphic.
\begin{conjecture}[$\ell^p$-coarse Baum--Connes conjecture]\label{conj:l^pCBC}
	The evaluation map 
	\[ ev\colon C^p_L(P_r(X))\to C^p(X)\] that maps $f$ to $f(0)$ induces a $K$-theoretic isomorphism
	$$ev_*\colon \lim_{r\to\infty}K_*(C^p_L(P_r(X)))\to K_*(C^p(X)).$$
\end{conjecture}

We note that the $\ell^p$-Roe algebra defined here differs slightly from the one introduced by Zhang and Zhou in \cite{Zhangzhoulp}. 
More specifically, the $\ell^p$-Roe algebra of $P$  in \cite[Definition 2.6]{Zhangzhoulp}, which we denote by $\mathbf{C}^p(P)$, is defined to be the Banach algebra generated by locally compact operators with finite propagation acting on $\ell^p(Z,\ell^p(\N))$, where $Z$ is a countable dense subset of $P$ as above. Similarly, there is a corresponding version of  $\ell^p$-localization algebra from \cite{Zhangzhoulp}, which we denote by $\mathbf{C}_L^p(P)$. A standard Mayer--Vietoris argument shows that the $K$-groups of $\mathbf{C}_L^p(P)$ (from \cite{Zhangzhoulp})  and $K$-groups of the $\ell^p$-localization algebra  $C_L^p(P)$ of the current paper  are naturally isomorphic when $P$ is a finite dimensional simplicial complex, both being isomorphic to the $K$-homology of $P$. However, whether the two definitions of $\ell^p$-Roe algebras have isomorphic $K$-groups in general remains unclear. More precisely, we state the question as follows. 
\begin{question}
	Let $X$ be a proper metric space and $Z$ a countable dense subset of $X$. Let $C^p(X)$ and $\mathbf C^p(X)$ be the completion of locally compact and finite propagation operators acting on the $C_0(X)$-modules $\ell^p(Z,\ell^p(\mathbb N))$ and $\ell^p(Z,\ell^p(\mathbb N,H))$, respectively. Let $e$ be a rank-one projection on $H$. Does the following map
	$$\mathbf C^p(X)\to C^p(X),~T\mapsto T\otimes e$$
	induces an isomorphism $K_*(\mathbf C^p(X))\to K_*(C^p(X))$?
\end{question}
\subsection{A vector-valued Marcinkiewicz--Zygmund inequality}
In this subsection, we prove a technical result concerning the operator norms of operators on $\ell^p$-spaces. Throughout this subsection, we assume that $X$ is a countable discrete set and $H=\ell^2(\N)$. Recall that the classic Marcinkiewicz--Zygmund inequality states that a bounded operator $T$ acting on $\ell^p(X)$ induces a bounded operator $T\otimes 1$ on $\ell^p(X,H)$ with the same norm. We will prove a vector-valued version of this inequality, as stated below.
\begin{theorem}\label{thm:VectorMarZygmundIneq}
	Let $T$ be a bounded operator acting on $\ell^p(X,H)$. Then there exists $K_G>0$ such that   
	$$\sum_{x\in X}\left(\sum_{i=1}^N\|(T\varphi_i)(x)\|_H^2\right)^{p/2}\leq K_G^p\|T\|^p\cdot\sum_{x\in X}\left(\sum_{i=1}^N\|\varphi_i(x)\|_H^2\right)^{p/2}$$
	for any  elements  $\varphi_1,\ldots,\varphi_N\in \ell^p(X,H)$.
	Equivalently, $T\otimes 1$ is a bounded operator on $\ell^p(X,H\otimes H)$ with $\|T\otimes 1\|\leq K_G\|T\|$.
\end{theorem}

The proof of Theorem \ref{thm:VectorMarZygmundIneq} is inspired by \cite[Theorem 8.1]{MR829919}. A key ingredient of the proof is the Grothendieck inequality, which we recall here for the convenience of the reader.
\begin{theorem}[Grothendieck inequality]\label{thm:GrothendieckIneq}
		Let $T\colon\ell^\infty(X)\to\ell^1(X)$ be a bounded operator. Then there exists $K_G>0$ such that    
	$$\sum_{x\in X}\left(\sum_{i=1}^N|(T\varphi_i)(x)|^2\right)^{1/2}\leq K_G\|T\|\cdot\sup_{x\in X}\left(\sum_{i=1}^N|\varphi_i(x)|^2\right)^{1/2}$$
for any $\varphi_1,\ldots,\varphi_N\in \ell^\infty(X)$.
	Equivalently, $T\otimes 1$ is a bounded operator from $\ell^\infty(X,H)$ to $\ell^1(X,H)$ with $\|T\otimes 1\|\leq K_G\|T\|$.
\end{theorem}

\begin{proof}[Proof of Theorem \ref{thm:VectorMarZygmundIneq}]
	We first briefly show that the equivalence between the desired inequality and the boundedness of $T\otimes 1$. Let $\{e_i\}_{i=1,2,\ldots}$ be an orthonormal basis of $H$. Consider an element
	$$\varphi=\sum_{i=1}^N\varphi_i\otimes e_i\in \ell^p(X,H\otimes H)$$
with  $\varphi_1,\ldots,\varphi_N\in \ell^p(X,H)$. We note that
	$$\|\varphi\|^p=\sum_{x\in X}\left(\|\sum_{i=1}^N\varphi_i(x)\otimes e_i\|_{H\otimes H}\right)^p=\sum_{x\in X}\left(\sum_{i=1}^N\|\varphi_i(x)\|_H^2\right)^{p/2}.$$
	Since
	$$(T\otimes 1)\varphi=\sum_{i=1}^NT\varphi_i\otimes e_i,$$
	we have
	$$\|(T\otimes 1)\varphi\|^p=\sum_{x\in X}\left(\sum_{i=1}^N\|(T\varphi_i)(x)\|_H^2\right)^{p/2}.$$
	All such $\varphi$'s as above  are dense in $\ell^p(X,H\otimes H)$. Therefore, we have $\|T\otimes 1\|\leq C$ if and only if 
	$$\sum_{x\in X}\left(\sum_{i=1}^N\|(T\varphi_i)(x)\|_H^2\right)^{p/2}\leq C^p\cdot\sum_{x\in X}\left(\sum_{i=1}^N\|\varphi_i(x)\|_H^2\right)^{p/2}$$
	for any $\varphi_1,\ldots,\varphi_N\in \ell^p(X,H)$.
	
	Now we prove the inequality above. Note that $\ell^p(X,H)=\ell^{p,2}(X\times \N)$, which consists of functions over $X\times\N$ equipped with the $\ell^p$-$\ell^2$ mixed norm, that is, each element $f\in \ell^{p,2}(X\times \N)$ is equipped with the norm
	\[   \left(\sum_{x\in X}\left(\sum_{n\in \N} |f(x, n)|^2\right)^{p/2}\right)^{1/p}.\] 
	Let $\varphi_1, \dots, \varphi_N \in \ell^p(X, H) = \ell^{p,2}(X\times \N)$ as above. We define $\varPhi$ to be the function on $X\times\N$ given by
	$$\varPhi(x,n)\coloneqq \left(\sum_{i=1}^N|\varphi_i(x,n)|^2\right)^{1/2},~\forall x\in X,~n\in \N.$$
	 Consider the multiplication operator
	$$M_\varPhi\colon \ell^\infty(X\times\N)\to \ell^{p,2}(X\times \N),~\xi\mapsto \xi\cdot\varPhi,$$
	where $\xi\cdot\varPhi$ stands for the pointwise multiplication. 
	For any $\xi\in \ell^\infty(X\times\N)$, we have
	\begin{align*}
		(\|M_\varPhi\xi\|_{\ell^{p,2}(X\times \N)})^p=&\sum_{x\in X}\left(\sum_{n=1}^\infty|\xi(x,n)\varPhi(x,n)|^2\right)^{p/2}\\
		\leq&\sum_{x\in X}\left(\sum_{n=1}^\infty\sum_{i=1}^N|\varphi_i(x,n)|^2\right)^{p/2}\cdot (\|\xi\|_{\ell^\infty(X\times\N)})^p.
	\end{align*}
	Therefore,
	$$\|M_\varPhi\|^p\leq \sum_{x\in X}\left(\sum_{i=1}^N\|\varphi_i(x)\|_H^2\right)^{p/2}.$$
	
	For $p\in [1,\infty)$, $\ell^{p'}(X,H)$ is the dual space of $\ell^p(X,H)$, where $p'$ satisfies  $1/p+1/p'=1$. For each  $\xi\in \ell^{p'}(X,H)=\ell^{p',2}(X\times \N)$ with $\|\xi\|_{\ell^{p',2}(X\times \N)}\leq 1$, we  define the multiplication operator
	$$M_\xi'\colon \ell^{p,2}(X\times \N)\to \ell^1(X\times \N),~\varphi\mapsto\varphi\cdot\xi.$$ 
	Note that 
	$$\|M'_\xi\varphi\|_{\ell^1(X\times \N)}=(|\varphi|,|\xi|),$$
	where $(\cdot,\cdot)$ denotes the pairing between $\ell^p(X,H)$ and its dual. Therefore we have $\|M'_\xi\|\leq 1$.
	
	Denote $\varphi'_i=\varphi_i/\varPhi$, where we define $\varphi'_i(x,n)=0$ if $\varPhi(x,n)=0$. Note that for any $(x,n)\in X\times\N$, we have
	$$\left(\sum_{i=1}^N|\varphi'_i(x,n)|^2\right)^{1/2}=1.$$
	Hence $\varphi_i'\in\ell^\infty(X\times\mathbb N)$.
	Since $M_\varPhi(\varphi'_i)=\varphi_i$, we have
	$$\sum_{x\in X}\sum_{n=1}^\infty\left(\sum_{i=1}^N|\left(M'_\xi TM_\varPhi(\varphi'_i)\right)(x)|^2\right)^{1/2}=\sum_{x\in X}\sum_{n=1}^\infty|\xi(x,n)|\left(\sum_{i=1}^N|(T\varphi_i)(x,n)|^2\right)^{1/2}.$$
	
	By applying the Grothendieck inequality (Theorem \ref{thm:GrothendieckIneq}) to the operator
	$$M'_\xi TM_\varPhi\colon \ell^\infty(X\times\N)\to\ell^1(X\times \N)$$
	and the functions $\varphi'_1,\ldots,\varphi'_N$ as above,  we  obtain that
	\begin{align*}
		\sum_{x\in X}\sum_{n=1}^\infty|\xi(x,n)|\left(\sum_{i=1}^N|(T\varphi_i)(x,n)|^2\right)^{1/2}\leq& K_G \|M_\varPhi\|\cdot\|T\|\cdot\|M'_\xi\|\\
		=&K_G\|T\|\left(\sum_{x\in X}\left(\sum_{i=1}^N\|\varphi_i(x)\|_H^2\right)^{p/2}\right)^{1/p}.
	\end{align*}
	Since $\xi$ is arbitrary, this finishes the proof.	
\end{proof}
\subsection{Proof of Theorem \ref{thm:lp-lqcBCIntro}}
In this subsection, we finish the proof of Theorem \ref{thm:lp-lqcBCIntro} for general $p\in[1,+\infty)$. The proof is completely similar to the proof of Theorem \ref{thm:lqcBC}. We first reduce the proof to the case when $X$ is a coarse disjoint  union of finite sets. We then use the same operator constructed in Section \ref{sec:BottDirac} and construct the $\ell^p$-version of the corresponding twisted Roe algebras and the index map as in Section \ref{sec:cbcTwist}. Finally, Theorem \ref{thm:lp-lqcBCIntro} follows similarly from the coarse Baum--Connes isomorphism for the $\ell^p$-twisted algebras. 

Let us start from the matrix decomposition for operators in $C^p_{alg}(P_r(X))$. Let $(B_{r,x})_{r\geq 0, x\in X}$ be the \emph{good cover system} on $P_r(X)$ as in Section \ref{sec:twistedRoe}. Recall that $C^p_{alg}(P_r(X))$ is defined as the collection of operators acting on the $\ell^p$-geometric module $\ell^p(Z_r(X),\ell^p(\N,H))$. For any element $T\in C^p_{alg}(P_r(X))$ and $x,y\in X$, we set
$$T(x,y)=(\chi_{B_{r,x}}\otimes 1) T(\chi_{B_{r,y}\otimes 1})\colon \ell^p(B_{r,y},\ell^p(\N,H))\to \ell^p(B_{r,x},\ell^p(\N,H)).$$
\begin{lemma}\label{lemma:lpPropagation}\ 
	\begin{enumerate}[label=$({\arabic*})$]
		\item $T$ is locally compact if and only if $T(x,y)$ is compact for any $x,y\in X$.
		\item $T$ has finite propagation if and only if there exists $r_1>0$ such that $T(x,y)=0$ if $d(x,y)>r$. Furthermore, the smallest $r_1$ satisfies
		$$|r_1-\prop(T)|\leq 2.$$
		\item There exists $C>0$ such that 
		$$\sup_{x,y\in X}\|T(x,y)\|\leq \|T\|\leq C\sup_{x,y\in X}\|T(x,y)\|,$$ where $C$ only depends on the smallest $r_1$ above and the bounded geometry of $X$.
	\end{enumerate}
\end{lemma}
\begin{proof}
	We only prove that $\|T\|\leq C\sup_{x,y\in X}\|T(x,y)\|$ for some $C>0$, and the rest of the lemma follows directly from  the definition.
	
	For any $\varphi\in \ell^p(Z_r(X),\ell^p(\N,H))$, we have 
	$$\varphi=\sum_{x\in X}\varphi_x,\text{ and }\|\varphi\|^p=\sum_{x\in X}\|\varphi_x\|^p,$$
	where $\varphi_x\coloneqq (\chi_{B_{r,x}}\otimes 1)\varphi$. Therefore,
	$$T\varphi=\sum_{x\in X}(T\varphi)_x=\sum_{x\in X}\sum_{y\in X}(T\varphi_y)_x=\sum_{x\in X}\sum_{y\in X}T(x,y)\varphi_y,$$
	and
	$$\|T\varphi\|^p=\sum_{x\in X}\left\|\sum_{y\in X}T(x,y)\varphi_y\right\|^p$$
	Denote $N_{r_1}(x)$ the $r_1$-neighborhood of $x$ in $X$, where $r_1$ is the number from part $(2)$. 
	Since $X$ has bounded geometry, there exists $C>0$ such that for any $x\in X$, the number of elements of x$X$ that lie in $N_{r_1}(x)$ is at most $C$. By assumption, we have
	\begin{align*}
		\|T\varphi\|^p\leq& \sup_{x,y\in X}\|T(x,y)\|^p\cdot \sum_{x\in X}\left\|\sum_{y\in N_{r_1}(x)}\varphi_y\right\|^p\\
		\leq &\sup_{x,y\in X}\|T(x,y)\|^p\cdot C^{p-1}\sum_{x\in X}\sum_{y\in N_{r_1}(x)}\|\varphi_y\|^p\\
		=&\sup_{x,y\in X}\|T(x,y)\|^p\cdot C^{p-1}\sum_{y\in X}\sum_{x\in N_{r_1}(y)}\|\varphi_y\|^p\\
		\leq &\sup_{x,y\in X}\|T(x,y)\|^p\cdot C^{p-1}\|\varphi\|^p\cdot C=\sup_{x,y\in X}\|T(x,y)\|^p\cdot C^{p}\|\varphi\|^p.
	\end{align*}
	This finishes the proof.
\end{proof}

Now we are ready to define an $\ell^p$-version of twisted algebras, an analogue of those introduced in Section \ref{sec:twistedRoe}. In particular, the discussion below is very similar to that from Section \ref{sec:twistedRoe}.  Let $X=\bigsqcup_n X_n$ be a coarse disjoint union of finite sets, and $f\colon X\to E=\bigoplus E_n$ a coarse embedding, where the $E_n$'s are finite dimensional Euclidean space equipped with the $\ell^q$-metric. Let $H_{E_n}\coloneqq L^2(E_n,\Bigwedge^*E_n)$ be the Hilbert space of differential forms over $E_n$ with respect to the Euclidean metric $\|\cdot\|_2$ on $E_n$, and $H_E=\oplus H_{E_n}$. We denote 
$$\mathscr M_{P_r(X),E}=\ell^q(Z_r(X),\ell^q(\N,H\otimes H_{E})),$$
namely the space of $\ell^p$-functions on $Z_r(X)$ with values in $\ell^q(\N,H\otimes H_{E})$. Similarly, since $X=\bigsqcup_n X_n$ is a coarse disjoint union of finite set,  we only consider block-wise diagonal operators acting on $\mathscr M_{P_r(X),E}$. 

The Banach space $\mathscr M_{P_r(X),E}$ is equipped with actions of Borel functions from both $P_r(X)$ and  $E$. Hence we define the support and propagation of bounded operators acting on $\mathscr M_{P_r(X),E}$ along both $P_r(X)$ and $E$, respectively.
For notational simplicity, we will denote $P=P_r(X)$ and $Z=Z_r(X)$.

We define the twisted Roe algebra as follows.
\begin{definition}\label{def:Aalgebralp}
	We define $A_{alg}^p(P,E)$ to be the collection of uniformly continuous paths of operators $\{T_s\}_{s\geq 1}$, where $T_s\coloneqq (T_s^n)_{n\in\N}$ is a block-wise diagonal operator acting on $\mathscr M_{P,E}$, such that $T_s(x,y)$ is compact for any $x,y\in X$ and furthermore the following properties hold.
	\begin{enumerate}
		\item $\|T_s(x,y)\|$ is uniformly bounded on $X\times X\times[1,+\infty)$.
		\item $T_s$ has uniformly finite propagation along $P$. Equivalently, there exists $r_1>0$ such that $T_s(x,y)=0$ for all $s\geq 1$, provided that  $d(x,y)>r_1$. 
		\item the $\ell^q$-propagation of $T_s(x,y)$ is uniformly finite and tends to zero uniformly as $s\to\infty$ for all $x, y\in X$,  and
		\item \label{Aalgebra5lp} there exists $r_2>0$ such that the $E$-support of $T_s(x,y)$ is contained in the $\ell^q$-ball centered at $f(y)$ with radius $r_2$ for all $x,y\in X$.
	\end{enumerate} 
	The norm of $\{T_s\}_{s\geq 1}$ is defined by 
	$\sup_{s\geq 1}\|T_s\|$,  where $\|\cdot\|$ denote the operator norm on $\mathscr M_{P,E}$. The completion of $A^p_{alg}(P,E)$ with respect to this norm is denoted by $A^p(P,E)$.
\end{definition}

Similarly, we define $A_L^p(P,E)$ to be the completion of the algebra of norm-bounded and norm-continuous maps $\mathbb T \colon [1,\infty) \to A^p_{alg}(P,E)$ such that   the propagation  of $\mathbb T(t)$ along $P$ tends to zero, as $t\to \infty$, where the completion is taken with respect to the norm
\[ \|\mathbb T\| =  \sup_{t\geq 1} \|\mathbb T(t)\|.\]

Let $C^p(P,E)$ be the $\ell^p$-Roe algebra associated with the $\ell^p$-geometric module $\mathscr M_{P,E}$, namely locally compact operators acting on $\mathscr M_{P,E}$ with finite propagation along $P$. Clearly, there is a homomorphism
$$\imath\colon A^p(P,E)\to C^p(P,E)$$
induced by evaluation at $s=1$ and  forgetting the last condition \eqref{Aalgebra5lp} in Definition \ref{def:Aalgebralp}. 

The key reason for defining the $\ell^p$-Roe algebra by operators acting on the module $\ell^p(Z,\ell^p(\N,H))$ as Definition \ref{def:lpRoe} is that the following map
$$\id_*\colon K_*(C^p(P))\to K_*(C^p(P,E))$$
induced by tensoring with a fixed rank-one projection is an isomorphism. In fact, the two Banach algebras $C^p(P)$ and $C^p(P,E)$ are isomorphic, since an isomorphism $H\cong H\otimes H_E$ induces an isomorphism between the two $\ell^p$- geometric modules $\ell^p(Z,\ell^p(\N,H))$ and $\mathscr M_{P,E}=\ell^p(Z,\ell^p(\N,H\otimes H_E))$. 

Similar to the proof of Theorem \ref{thm:lqcBCIntro}, the proof of Theorem \ref{thm:lp-lqcBCIntro} consists of two parts: the isomorphism 
\[ \id_*\colon K_*(C^p(P))\to K_*(C^p(P,E)) \] factors through $K_*(A^p(P,E))$, and the coarse Baum--Connes conjecture holds for $A^p(P,E)$. To prove the first part, we similarly construct an $\ell^p$-version of the index map
$$\ind\colon K_*(C^p(P))\to K_*(A^p(P,E)).$$

Let $S=(S_n)_{n\in \N}$, where $S_n\colon X_n\to H\otimes H_{E_n}$ is a uniformly bounded operator-valued function on $X_n$. We set $1\otimes S$ to be the operator acting on $\mathscr M_{P,E}$. Note that $1\otimes S$ is bounded with $\|1\otimes S\|\leq \sup_{x\in X}\|S(x)\|$, and $1\otimes S$ has zero propagation. For simplicity, we still write $S$ for $1\otimes S$.
Recall the operator $\Phi_{s,v}$  from Proposition \ref{prop:phi}, and the path of block-wise diagonal operators $\Phi = (\Phi_s)_{s\geq 1}$ with 
$$\Phi_s\colon H_{P,E}\to H_{P,E}$$
given by
$$\Phi_s\xi=\Phi_{s+s_n,f(x)}(\xi) \textup{ for all } x\in X_n \text{ and }\xi\in (\chi_{B_x}\otimes 1)\mathscr M_{P,E}. $$
With respect to the even-odd grading on each $\Bigwedge^*E_n$, we have
$$\Phi=\begin{pmatrix}
	0&\Phi_-\\
	\Phi_+&0
\end{pmatrix}.$$

Let $\psi_{v,R}$ be the cut-off function  from Lemma \ref{lemma:cut-off}. For each $R>0$, we define a bounded operator 
$$\psi_R\colon H_{P,E}\to H_{P,E} $$
by 
$$\psi_R\xi=\psi_{f(x),R}\xi\textup{ for all } \xi\in (\chi_{B_x}\otimes 1)\mathscr M_{P,E}.$$

Now we consider a bounded operator $T$ acting on $\ell^p(P,\ell^p(\N,H))$. The operator $T\otimes 1$ acting on $\mathscr M_{P,E}=\ell^p(P,\ell^p(\N,H\otimes H_E))$ is still bounded and $\|T\otimes 1\|\leq K_G\|T\|$ by Theorem \ref{thm:VectorMarZygmundIneq}. Clearly $T\otimes 1$ is in general not locally compact, but has the same support and propagation along $P$ as the operator $T$. If there is no confusion arises, we will write $T$ for $T\otimes 1$.

Here we shall only give the construction of the map $\ind$ for $K_0$ groups, since  the case of $K_1$ groups follows easily by a standard  suspension argument. 
Let $\alpha$ be an element of $K_0( C^p(P))$. Then $\alpha=[\cP] - [\cP']$, where $\cP$ and $\cP'$ are idempotents in $\mathscr M_n(C^p(P)^+)$ such that  $\cP-\cP'\in \mathscr M_n(C^p(P))$. We may assume that $\cP\in C^p(P)$ without loss of generality, and explain the construction of the index map on the element $\cP$.

Let  $\cP\otimes 1\colon \mathscr M_{P,E} \to \mathscr M_{P,E}$ be the operator constructed from the operator $\cP$. For simplicity, we shall write $\cP$ in place of $\cP\otimes 1$ below.
Let us consider the following idempotent.
\begin{equation}\label{eq:indexlp}
	\cP_\Phi\coloneqq \begin{pmatrix}
		\cP-(\cP-\cP\Phi_+\cP\Phi_-)^2&(2\cP-\cP\Phi_+\cP\Phi_-)\cP\Phi_+(\cP-\cP\Phi_-\cP\Phi_+)\\ \cP\Phi_-(\cP-\cP\Phi_+\cP\Phi_-)&(\cP-\cP\Phi_-\cP\Phi_+)^2
	\end{pmatrix}
\end{equation}
where the above expression follows the standard formula for  the index class of the generalized Fredholm operator $\cP\otimes \Phi$.  For $R>0$, we define
\begin{align*}
	\cP_{\Phi,R}\coloneqq &\begin{pmatrix}
		\cP-\psi_R(\cP-\cP\Phi_+\cP\Phi_-)^2\psi_R&\psi_R(2\cP-\cP\Phi_+\cP\Phi_-)\cP\Phi_+(\cP-\cP\Phi_-\cP\Phi_+)\psi_R\\\psi_R\cP\Phi_-(\cP-\cP\Phi_+\cP\Phi_-)\psi_R&\psi_R(\cP-\cP\Phi_-\cP\Phi_+)^2\psi_R
	\end{pmatrix}.
\end{align*}
The following proposition shows that the map 
\[  [\cP] \mapsto [\cP_{\Phi,R}]-[\begin{psmallmatrix}
	\cP & 0 \\ 0 & 0 
\end{psmallmatrix}]\]
gives a well-defined homomorphism
$$\ind\colon K_*(C^*(P))\to K_*(A^p(P,E)).$$
\begin{proposition}\label{prop:indMaplp}
	If $\cP\in C^p(P)$ and $\cP_{\Phi,R}$ are as above, then
	$$\cP_{\Phi,R}-\begin{psmallmatrix}
		\cP & 0 \\ 0 & 0 
	\end{psmallmatrix}\in A(P,E)$$
	and there is $C>0$ depending only on $\|\cP\|$ such that $\|\cP_{\Phi,R}\|\leq C$. Furthermore, if $\cP\in C^p(P)$ and $\cP^2=\cP$, then there exist $R>0$ and a function $\Phi$ as in Proposition \ref{prop:phi} depending only on $\cP$ such that
	$$\|(\cP_{\Phi,R})^2-\cP_{\Phi,R}\|< \frac 1 4.$$
	Hence the formal difference $[P_{\Phi,R}]-[\begin{psmallmatrix}
		P & 0 \\ 0 & 0 
	\end{psmallmatrix}]$ defines a class in $K_0(A^p(P,E))$, denoted by $\ind([\cP])$.
\end{proposition}
\begin{proof}
	The norm estimates for $\cP_{\Phi,R}$ and $\cP_{\Phi}$ follow from the construction of $\cP_{\Phi,R}$ and $\cP_{\Phi}$, part (1) of Proposition \ref{prop:phi} and Theorem \ref{thm:VectorMarZygmundIneq}.	By Proposition \ref{prop:phi}, we have
	 $\psi_R\cP(\Phi^2-1)\psi_R\in A^p(P,E)$ and $\psi_R[\cP,\Phi]\psi_R\in A^p(P,E)$. It follows that  $\cP_{\Phi,R}-\begin{psmallmatrix}
		\cP & 0 \\ 0 & 0 
	\end{psmallmatrix}\in A(P,E)$.
	
	Since $C^p_{alg}(P)$ is dense in $C^p(P)$, for any $\delta>0$, there exists $\cQ\in C^p_{alg}(P)$ such that $\|\cQ\|=\|\cP\|$ and $\|\cP-\cQ\|<\delta(2\|\cP\|+1)$. In particular, we have 
	\begin{itemize}
		\item $\sup_{x,y\in X}|\cQ(x,y)|<\infty$,
		\item $\|\cQ^2-\cQ\|<\delta$.
		\item the propagation of $\cQ$ is finite, hence $\cQ(x,y)=0$ if $d(x,y)>r_1$ for some $r_1>0$,
	\end{itemize}
	To show that $\|(\cP_{\Phi,R})^2-\cP_{\Phi,R}\|< 1/4$, it suffices to estimate $\|\cQ_{\Phi,R}-\cQ_\Phi\|$, where  
	\[  	\cQ_\Phi\coloneqq \begin{pmatrix}
		\cQ-(\cQ-\cQ\Phi_+\cQ\Phi_-)^2&(2\cQ-\cQ\Phi_+\cQ\Phi_-)\cQ\Phi_+(\cQ-\cQ\Phi_-\cQ\Phi_+)\\ \cQ\Phi_-(\cQ-\cQ\Phi_+\cQ\Phi_-)&(\cQ-\cQ\Phi_-\cQ\Phi_+)^2
	\end{pmatrix} \]
	and $\cQ_{\Phi,R}$ is defined by the same formula  as $\cP_{\Phi, R}$ but with $\cP$ replaced by $\cQ$.

	The matrix decomposition of $[\cQ,\psi_R]$ is given by
	$$\left([\cQ,\psi_R]\right)(x,y)=\cQ(x,y)(\psi_{f(x), R}-\psi_{f(y), R}).$$
	In particular, if  $\left([\cQ,\psi_R]\right)(x,y)\ne 0$, then 
	$$\|f(x)-f(y)\|_q\leq \rho_+(r_1)$$
	where $\rho_+$ is the function as given in Definition  \ref{def:coarseembed}. Therefore by Lemma \ref{lemma:cut-off}, we have
	$$\|\left([\cQ,\psi_R]\right)(x,y)\|\leq 3R^{-1}\rho_+(r_1).$$
	Therefore, by part (3) of Lemma \ref{lemma:lpPropagation}, we have
	$$[\cQ,\psi_R]\to 0\text{ as }R\to\infty.$$
	Furthermore, it follows from Lemma \ref{lemma:cut-off} and part (3) of Proposition \ref{prop:phi} that
	$$[\Phi,\psi_R]\to 0\text{ as }R\to\infty.$$
	Similarly, the matrix decomposition of  $(1-\psi_R)[\cQ,\Phi]$ is given by 
	$$\left( (1-\psi_R)[\cQ,\Phi]\right) (x, y)=(1-\psi_{f(x), R})(\Phi_{f(x)}-\Phi_{f(y)}) \cQ(x,y),$$
	Since $\cQ$ has finite propagation, it follows from part  (5)  of Proposition \ref{prop:phi} and part (3) of Lemma \ref{lemma:lpPropagation} that  $\|(1-\psi_R)[\cQ,\Phi]\|$ can be made arbitrarily small as long as  the number $\varepsilon$ from Proposition \ref{prop:phi} is chosen to be sufficiently small and $R$ is chosen to be sufficiently large. Now the proof of this proposition follows from the same computation in the proof of Proposition \ref{prop:indMap}.
\end{proof}

We complete the proof of Theorem \ref{thm:lp-lqcBCIntro}. 
\begin{proof}[Proof of Theorem \ref{thm:lp-lqcBCIntro}]
	The proof of Theorem \ref{thm:lp-lqcBCIntro} is completely similar to the proof of Theorem \ref{thm:lqcBC}. We shall be brief. By the same construction and Eilenberg swindle argument  in Proposition \ref{prop:composition}, we derive that the map 
	$$\imath_*\circ\ind\colon K_*(C^p(P))\to K_*(A^p(P,E))\to K_*(C^p(P,E))$$
	is equal to the isomorphism $\id_*\colon K_*(C^p(P))\to  K_*(C^p(P,E))$ induced by taking the tensor product with a rank-one projection. Now Theorem \ref{thm:lp-lqcBCIntro} reduces to the coarse Baum--Connes isomorphism for the twisted map
	$$\ev_*\colon \lim_{r\to\infty}K_*(A_L^p(P_r(X),E))\to K_*(A^p(P_r(X),E)).$$
	The proof of the fact that the above $\ev_*$ is an isomorphism  follows from the same construction and the Mayer--Vietoris argument as the $\ell^2$-case in Section \ref{sec:cbcTwist}. This finishes the proof of Theorem \ref{thm:lp-lqcBCIntro}.
\end{proof}

As a corollary of Theorem \ref{thm:lp-lqcBCIntro}, we see that for different $p\in[1,+\infty)$, the $K$-theory of the $\ell^p$-Roe algebra $K_*(C^p(X))$ is independent of $p$, as long as $X$ coarsely embeds into $\ell^q$.


\end{document}